\documentclass{article}
\usepackage[utf8]{inputenc}
\usepackage{amsmath}
\usepackage{amssymb}
\usepackage{amsthm}
\usepackage{enumitem}

\newtheorem{corollary}{Corollary}
\newtheorem{lemma}{Lemma}
\newtheorem{remark}{Remark}
\newtheorem{example}{Example}
\newtheorem{proposition}{Proposition}
\newtheorem{definition}{Definition}
\newtheorem{conjecture}{Conjecture}
\newtheorem{theorem}{Theorem}

\DeclareMathOperator{\toPalCouple}{toPalCpl}
\DeclareMathOperator{\Suffix}{Suf} %
\DeclareMathOperator{\Prefix}{Prf} %
\DeclareMathOperator{\Factor}{Fac} %
\DeclareMathOperator{\Pal}{Pal} %
\DeclareMathOperator{\ctPC}{ctPC} %
\DeclareMathOperator{\OrdFactor}{OrdFac} 
\DeclareMathOperator{\Spread}{Spread}

\DeclareMathOperator{\NPSCover}{NPSCover}
\DeclareMathOperator{\minNPSCov}{minNPSCov}
\DeclareMathOperator{\separate}{separate}
\DeclareMathOperator{\NPP}{NPP} 
\DeclareMathOperator{\width}{width}
\DeclareMathOperator{\height}{height}

\DeclareMathOperator{\mper}{mper}
\DeclareMathOperator{\order}{order}

\DeclareMathOperator{\PL}{PL}
\DeclareMathOperator{\PPL}{PPL}
\DeclareMathOperator{\Close}{Close}

\DeclareMathOperator{\bottom}{bottom}

\DeclareMathOperator{\maxPL}{maxPL}
\DeclareMathOperator{\maxPPL}{maxPPL}
\DeclareMathOperator{\Mirror}{Mirror}
\DeclareMathOperator{\add}{add}
\DeclareMathOperator{\PalExt}{PalExt}

\DeclareMathOperator{\Run}{Run}
\DeclareMathOperator{\UC}{UC} 

\DeclareMathOperator{\RunPal}{RunPal}
\DeclareMathOperator{\minRunPal}{minRunPal}
\DeclareMathOperator{\Period}{Period}
\DeclareMathOperator{\FirmPalPrefix}{FirmPalPrf}

\DeclareMathOperator{\pToRun}{pToRun}
\DeclareMathOperator{\perProlong}{perPrl}
\DeclareMathOperator{\CovPalAll}{cpAll}
\DeclareMathOperator{\CovPalAllLeft}{cpAllLeft}
\DeclareMathOperator{\CovPalAllRight}{cpAllRight}
\DeclareMathOperator{\CovPalEdgeLeft}{cpEL}
\DeclareMathOperator{\CovPalEdgeRight}{cpER}
\DeclareMathOperator{\CovPalEdge}{cpE}
\DeclareMathOperator{\CovPalCmd}{cpELCmd}

\DeclareMathOperator{\PalCouple}{PalCpl}
\DeclareMathOperator{\NestPer}{NPS} 
\DeclareMathOperator{\diam}{diam} %
\DeclareMathOperator{\Cut}{Cut} %
\DeclareMathOperator{\pad}{pad}

\newcommand{\Cover}{V}
\newcommand{\BB}{B} 
\newcommand{\EE}{E} 
\newcommand{\zz}{z}
\newcommand{\bb}{b} %
\newcommand{\ww}{w}
\newcommand{\kk}{k}

\newcommand{\xx}{x}
\newcommand{\YY}{Y}
\newcommand{\hh}{h} %

\begin{document}

\title{Palindromic length of infinite aperiodic words}

\author{Josef Rukavicka\thanks{Department of Mathematics,
Faculty of Nuclear Sciences and Physical Engineering, Czech Technical University in Prague
(josef.rukavicka@seznam.cz).}}

\date{\small{October 16, 2024}}

\maketitle

\begin{abstract}
The palindromic length of the finite word $v$ is equal to the minimal number of palindromes whose concatenation is equal to $v$. It was conjectured in 2013 that for every infinite aperiodic word $x$, the palindromic length of its factors is not bounded.
We prove this conjecture to be true.
\end{abstract}

\section{Introduction}

Let \(u=u_1u_2\cdots u_n\) be a nonempty word of length \(n\), where \(u_i\) are letters and \(i\in\{1,2,\dots,n\}\). We say that \(u\) is a \emph{palindrome} if \(u_1u_2\cdots u_n=u_nu_{n-1}\cdots u_1\). For example, the words ``noon'' and ``level'' are palindromes. 
The \emph{palindromic length} \(\PL(u)\) of the finite nonempty word \(u\) is equal to the minimal number \(\kk\) such that \(u=p_1p_2\cdots p_{\kk}\), where \(p_i\) are palindromes and \(i\in \{1,2,\dots,\kk\}\). 

Let \(\xx\) be an infinite word. If there are finite words \(u,v\) such that \(\xx=uv^{\infty}\) then we say that \(\xx\) is \emph{ultimately periodic}. If \(\xx\) is not ultimately periodic then we say that \(\xx\) is \emph{aperiodic}.
In 2013, Frid, Puzynina, and Zamboni conjectured that \cite{FrPuZa}:
\begin{conjecture}
\label{tue8eiru883}
If \(\xx\) is an infinite word and \(\kk\) is an integer such that \(\PL(u)\leq \kk\) for every factor \(u\) of \(\xx\) then \(\xx\) is ultimately periodic. 
\end{conjecture}
It was shown in \cite{10.1007/978-3-319-66396-8_19} that for an infinite aperiodic word \(x\), we have that the palindromic length of factors of \(x\) is unbounded if and only if the palindromic length of prefixes of \(x\) is unbounded. 
Several partial solutions of Conjecture \ref{tue8eiru883} can be found. In \cite{FrPuZa}, the conjecture has been proven for infinite \(\alpha\)-power free words, where \(\alpha\) is any positive integer. In \cite{FRID2018202}, the conjecture has been proven for Sturmian words. 
Also various properties of palindromic length have been investigated; see, for instance, \cite{AMBROZ201974,10.1007/978-3-031-05578-2_6,10.1007/978-3-030-62536-8_14,RUKAVICKA2022106,10.1007/978-3-319-66396-8_19}.  Concerning  algorithmical and computational aspects of the palindromic length, the reader may find some results in \cite{borozdin_et_al:LIPIcs:2017:7338,FICI201441,10.1007/978-3-319-07566-2_16,RuSh15}.

In the current article, we prove Conjecture \ref{tue8eiru883} to be true. We sketch briefly our proof. 
Let \(\mathbb{N}_0\), \(\mathbb{N}_1\), \(\mathbb{Z}\), and \(\mathbb{Q}\) denote the set of all non-negative integers, all positive integers, all integers, and all rational numbers, respectively. Given \(n_1\leq n_2\in\mathbb{Z}\), let \(\mathbb{N}(n_1,n_2)=\{n\in\mathbb{N}_0\mid n_1\leq n\leq n_2\}\).  

Because of a technical convenience in some proofs, we consider the special letter \(\bb\) that we use for ``padding'' as follows (see also Remark \ref{dhj8f7ekv2} for more details):  
Let \(\Sigma_0\) and \(\Sigma\) be two finite alphabets such that \(\Sigma_0\cup\{\bb\}=\Sigma\) and \(\bb\not\in\Sigma_0\).  
Given an infinite word \(x_0\in\Sigma_0^{\infty}\), let \(\pad(x_0)=x\in\Sigma^{\infty}\) be such that \(x[2i-1]=\bb\) and \(x[2i]=x_0[i]\), where \(i\in\mathbb{N}_1\). Given a finite word \(u\in\Sigma_0^+\), let \(\pad(u)=u_1\bb u_2\bb \dots \bb u_n\in\Sigma^+\), where \(u=u_1u_2\cdots u_n\), \(n=\vert u\vert\), and \(u_1,u_2,\dots,u_n\in\Sigma_0\). We call \(\pad(x_0)\) and \(\pad(u)\) \emph{padded} words. Note that every second letter in padded words is the letter \(\bb\).

For \(u\in\pad(\Sigma_0^+)\) we define the \emph{padded palindromic length} \(\PPL(u)\) to be the minimal number \(\kk\) such that \(u= p_1\bb p_2\bb \cdots \bb p_{\kk}\), where \(p_1,p_2,\dots,p_{\kk}\in\Sigma^+\) are palindromes. 
Given \(x_0\in\Sigma_0^{\infty}\), let \[\begin{split}\maxPL(x_0)&=\max\{\PL(u)\mid u\in\Sigma_0^+\mbox{ is a factor of }x_0\}\mbox{ and }\\
\maxPPL(x_0)&=\max\{\PPL(\pad(u))\mid u\in\Sigma_0^+\mbox{ is a factor of }x_0\}\mbox{.}\end{split}\] Obviously \(u\in\Sigma_0^+\) is a palindrome if and only if \(\pad(u)\in\Sigma^+\) is a palindrome. Then the following lemma is easy to be verified. We omit the proof.
\begin{lemma}
\label{dd554hye88rc}
Suppose \(x_0\in\Sigma^{\infty}\). We have that  
\(\maxPL(x_0)<\infty\) if and only if \(\maxPPL(x_0)<\infty\).
\end{lemma}

Given a nonempty word \(t\in\Sigma^+\), let \(\mper(t)\) denote the minimal period of \(t\) and let \(\order(t)=\frac{\vert t\vert}{\mper(t)}\in\mathbb{Q}\). We say that a palindrome \(p\in\Sigma^+\) is \emph{non-periodic} if \(\order(p)<2\). 
Given a word \(t\in\Sigma^+\), let \(\NPP(t)\) denote the set of all nonempty non-periodic palindromic prefixes of \(t\).
We say that a word \(t\in\Sigma^+\) is \emph{ordinary} if for every factor \(u\) of \(t\) we have that \(\vert \NPP(t)\vert\geq\vert\NPP(u)\vert\). 

Suppose \(\ww_0\in\Sigma_0^{\infty}\) to be an infinite aperiodic word such that \(\maxPL(\ww_0)<\infty\) and \(\maxPPL(\ww_0)= \kk<\infty\). Let \(\ww=\pad(\ww_0)\in\Sigma^{\infty}\).
Let \(c_1=5\), \(c_2=8\), and \(c_3=10\) be constants. Given \(\hh,m\in\mathbb{N}_1\), 
let \(\lambda(\hh,m)=c_2^m(2c_1c_3\hh)^{m^2}\in\mathbb{N}_1\) be a function.
Let \(\hh_0\in\mathbb{N}_1\) be such that if \(\hh\geq \hh_0\) then 
\begin{equation}\label{adhuf99j}2^{\hh-1}>\kk(c_3\hh)^m\lambda(\hh,m)\mbox{ for all }m\in\mathbb{N}(1,\kk)\mbox{.}\end{equation} Obviously such \(\hh_0\) exists, since \(\lim_{\hh\rightarrow \infty}\frac{\kk(c_3\hh)^{\kk}\lambda(\hh,\kk)}{2^{\hh}}=0\mbox{.}\) 

We show that \(\ww\) contains an ordinary palindromic factor \(\zz\) such that \(\bb\) is a prefix of \(\zz\) and \(\vert\NPP(\zz)\vert\geq \hh_0\); see Lemma \ref{yyie9d33b}. 
Given \(m\in\mathbb{N}_1\), let 
\[\Cover(m)=\{n\in\mathbb{N}(1,\vert \zz\vert)\mid \bb=\zz[n]\mbox{ and }\PPL(\zz[2,n-1])=m\}\mbox{.}\] 

Let \(\hh=\vert\NPP(\zz)\vert\). Given \(m\in\mathbb{N}_0\), let \(\theta(m)=(2c_1c_3\hh)^{m}\in\mathbb{N}_1\) be a function.
Let \(D\subseteq \mathbb{N}_1\), \(\xi\in\mathbb{N}_1\), and \(\overline D\subseteq D\). If \(\overline D=\emptyset\) or \(\max(\overline D)-\min(\overline D)+1\leq \xi\) then we call \(\overline D\) a \(\xi\)-\emph{cut} of \(D\). 
Given \(m\geq 1\), let \(\NestPer(m)\) be the set of all couples \((D,\xi)\) such that \(D\subseteq \mathbb{N}(1,\vert \zz\vert)\), \(D\not=\emptyset\), \(\xi\) is a period of \(\zz[\min(D),\max(D)]\),  and for every \(\xi\)-cut \(\overline D\) of \(D\) there is \(M\subseteq \NestPer(m-1)\) such that \(\vert M\vert\leq \theta(m)\) and \(\overline D\subseteq \bigcup_{(C,\xi_1)\in M}C\).
The formal definition of \(\NestPer(m)\) including \(\NestPer(0)\) will be stated in Section \ref{sec8903jhd}. 
We call \((D,\xi)\in\NestPer(m)\) a \emph{nested periodic structure} (NPS) and we call the set \(D\) an NPS \emph{cluster} of \emph{degree} \(m\). The important notion here is the ``nesting'' of NPS (the iterative definition of \(\NestPer(m)\) by means of \(\NestPer(m-1)\)) and the fact that every \(\xi\)-cut of an NPS cluster of degree \(m\) can be covered by a bounded number of NPS clusters of degree \(m-1\).

Given \(m\in\mathbb{N}_1\), we show that there is \(M\subseteq \NestPer(m)\) such that 
\begin{equation}\label{djuf7831c}\Cover(m)\subseteq\bigcup_{(D,\xi)\in M}D\mbox{ and }\vert M\vert\leq (c_3\hh)^{m}\mbox{.}\end{equation} 
In other words, we show that \(\Cover(m)\) can be covered by a bounded number of NPS clusters of degree \(m\). The proof of this result takes the most of the current article (Section \ref{sec8903jhd} and Lemma \ref{ry88dbh3v4}). In the proofs, we use the fact that  \(\zz\) is an ordinary factor  and hence \(\hh\) forms an upper bound on the number of non-periodic palindromic prefixes of all suffixes of \(\zz\).

Let \(\zz_1,\zz_2,\dots,\zz_{\hh}\) be the nonempty non-periodic palindromic prefixes of \(\zz\) such that \(\vert \zz_1\vert<\vert \zz_2\vert<\dots<\vert \zz_{\hh}\vert\) and \(\zz_{\hh}=\zz\). Since \(\hh=\vert\NPP(\zz)\vert\), we have that \(\zz_i\) exist and are unique. Realize that \(\zz_{i+1}\) has two ``natural'' occurrences of \(\zz_{i}\); namely as a prefix and a suffix. We call the starting positions of these occurrences \emph{base positions}.  
Consequently \(\zz_{i+2}\) has two base positions of \(\zz_{i+1}\) and four base positions of \(\zz_{i}\). It follows that \(\zz_{\hh}\) has \(2^{\hh-1}\) base positions of \(\zz_1\). If \(g\in\mathbb{N}(1,\hh)\) and \(e\in\mathbb{N}(1,\vert \zz\vert)\) is a base position of \(\zz_g\) then \(\zz[e,e+\vert \zz_g\vert-1]=\zz_g\). Let \(\widetilde\BB\) denote the set of all base positions \(e\in\mathbb{N}(1,\vert\zz\vert)\) of the palindrome \(\zz_1\). In Section \ref{cbhd88e7jf} we state a formal definition of base positions.
Using (\ref{djuf7831c}) we show that \begin{equation}\label{ddh7923vbf}\vert\widetilde\BB\cap\bigcup_{m=1}^{\kk}\Cover(m)\vert\leq \kk(c_3\hh)^{\kk}\lambda(\hh,\kk) \mbox{.}\end{equation} 
The essential observation in the proof is that if \((D,\xi)\in\NestPer(m)\) then there is \({\EE}\subset\mathbb{N}(1,\vert \zz\vert)\) such that \(\vert{\EE}\vert\leq {c_2}\) and \[D\cap\widetilde\BB\subseteq D\cap\widetilde \BB\cap\bigcup_{\delta\in{\EE}}\mathbb{N}(\delta,\delta+\xi-1)\mbox{.}\] Hence, for the intersection of an NPS cluster with base positions, we can consider only \(c_2\) periods of the nested periodic structure in question; see Proposition \ref{duid0jfgh}. This is due to the fact that the palindromes \(\zz_i\) are non-periodic and are ``hierarchically'' (like a binary tree) organised in the ordinary  factor \(\zz\). In consequence of the nesting of NPS we iteratively apply this observation for all the ``nested NPS''.

Then (\ref{adhuf99j}) and  (\ref{ddh7923vbf}) will allow us to show that \(\widetilde \BB\cap\bigcup_{m=1}^{\kk}\Cover(m)\subset \widetilde \BB\), since the number of base positions grows exponentially with \(\hh\) and the size of \(\widetilde\BB\cap\Cover(m)\) grows only polynomially with \(\hh\). It follows that we cannot cover all base positions with the set \(\bigcup_{m=1}^{\kk}\Cover(m)\), which  is a contradiction to our assumption \(\maxPPL(\ww_0)=\kk\). We conclude that \(\maxPPL(\ww_0)=\infty\), \(\maxPL(\ww_0)=\infty\), and consequently we prove Conjecture \ref{tue8eiru883} to be true; see Corollary \ref{djj887jejf}.

\section{Preliminaries}
Let \(\epsilon\) denote the empty word.
Let \(\Sigma^+\) denote the set of all nonempty finite words over \(\Sigma\) and let \(\Sigma^*=\{\epsilon\}\cup\Sigma^+\).
Let \(\Pal^+\subseteq \Sigma^+\) denote the set of all nonempty palindromes. Let \(\Pal^*=\{\epsilon\}\cup\Pal^+\). Let \(\Sigma^{\infty}=\{a_1a_2\cdots \mid a_i\in\Sigma\mbox{ and }i\in\mathbb{N}_1\}\) denote the set of all infinite words.

Given \(t\in\Sigma^*\), let \(\Factor(t)\), \(\Prefix(t)\), and \(\Suffix(t)\) denote the set of all factors, all prefixes, and all suffixes of \(t\), respectively. Note that \(\epsilon,t\in\Factor(t)\cap\Prefix(t)\cap\Suffix(t)\). 
Given \(x\in\Sigma^{\infty}\), let \(\Factor(x)\) and \(\Prefix(x)\) denote the set of all finite factors and all prefixes of \(x\), respectively.

Given \(t\in\Sigma^+\), let \(t^R=a_na_{n-1}\cdots a_1\), where \(n=\vert t\vert\), \(a_i\in\Sigma\), \(i\in\mathbb{N}(1,n)\), and \(t=a_1a_2\cdots a_n\). We call \(t^R\) a \emph{reverse} of \(t\). Note that \(t\in\Pal^+\) if and only if \(t=t^R\). In addition, we define \(\epsilon^R=\epsilon\).

Given \(t\in\Sigma^+\), we have that \(t=a_1a_2\dots a_{\vert t\vert}\), where \(a_i\in\Sigma\) and \(i\in\mathbb{N}(1,\vert t\vert)\). If \(i\leq j\in\mathbb{N}(1,\vert t\vert)\), then let \(t[i,j]=t_{i}t_{i+1}\cdots t_{j}\) and let \(t[i]=t_i\).

Given  \(\xi\in\mathbb{N}_1\) and \(S\subset\mathbb{Q}\), let 
\(\Spread(S,\xi)=\{i+\alpha\xi\in\mathbb{Q}\mid i\in S\mbox{ and } \alpha\in\mathbb{Z}\}\mbox{.}\)
We call \(\Spread(S,\xi)\) a \emph{spread} of \(S\).

Given \(D\subset\mathbb{Z}\) and \(j\in\mathbb{Z}\), let \(\add(D,j)=\{i+j\in\mathbb{Z}\mid i\in D\}\).

Given \(n_1\leq n_2\in\mathbb{N}_1\), \(j\in\mathbb{N}(n_1,n_2)\), and \(D\subset\mathbb{N}_1\), let \(\Mirror(n_1,n_2,j)=\overline j\in\mathbb{N}(n_1,n_2)\) be such that \(j-n_1=n_2-\overline j\) and let \(\Mirror(n_1,n_2,D)=\{\Mirror(n_1,n_2,j)\mid j\in D\cap  \mathbb{N}(n_1,n_2)\}\mbox{.}\) Note that if \(t\in\Sigma^+\) and \(t[n_1,n_2]\in\Pal^+\) then \(t[j]=t[\overline j]\).


Given \(n_1,n_2\in\mathbb{Z}\), let \(\mathbb{N}^2(n_1,n_2)=\{(n_3,n_4)\mid n_3\leq n_4\in\mathbb{N}(n_1,n_2)\}\).

Given \(D\subseteq \mathbb{Z}\), let \[\diam(D)=\begin{cases}
    0 &\mbox{ if }D=\emptyset\mbox{,}\\
    \max(D)-\min(D)+1 &\mbox{ if }D\not=\emptyset\mbox{,}
\end{cases}\]
and let \[\Close(D)=\begin{cases}
    \emptyset &\mbox{ if }D=\emptyset\mbox{,}\\
    \mathbb{N}(\min(D),\max(D)) &\mbox{ if }D\not=\emptyset\mbox{.}
\end{cases}\]
We call \(\diam(D)\) the \emph{diameter} of the set \(D\).

Given \(t\in\Sigma^+\), let \[\begin{split}\Period(t)=\{\xi\in\mathbb{N}_1\mid \mbox{ if }i\in{N}_1\mbox{ and }i+\xi\leq\vert t\vert\mbox{ then }t[i]=t[i+\xi]\}\mbox{.}\end{split}\] 
\begin{remark}
Note that if \(t\in\Sigma^+\) and \(\xi\geq\vert t\vert\) then \(\xi\in\Period(t)\).
\end{remark}

\section{Known results}
\begin{lemma}\label{z7d8djje} (see \cite[Chapter 8, Lemma 8.1.2]{Lothaire_2002}). 
Let $u$, $v$, $w$ be words such that $uv$ and $vw$ have the period $j$ and $\vert v\vert\geq j$. Then $uvw$ has the period $j$.
\end{lemma}

We present an obvious consequence of Lemma \ref{z7d8djje}. We omit the proof.
\begin{lemma}
\label{uu78rbg68}
If \(u\in\Sigma^+\), \(p\in\Pal^+\), \(\xi\in\Period(up)\), and \(\xi\leq \vert p\vert\) then \(\xi\in\Period(upu^R)\).
\end{lemma}

\begin{lemma} (see \cite[Lemma 1]{10.1007/978-3-662-46078-8_24})
\label{tudjkdi8545}
Suppose \(j\) is a period of a nonempty palindrome \(x\); then there are
palindromes \(p_1\) and \(p_2\) such that \(\vert p_1p_2\vert = j\), \(p_2\not=\epsilon\), and \(x = (p_1p_2)^*p_1\).
\end{lemma}

\begin{lemma} (see \cite[Lemma 2]{10.1007/978-3-662-46078-8_24})
\label{id8ieubmzmfj}
Suppose \(t\) is a palindrome and \(u\) is its proper suffix-palindrome or
prefix-palindrome; then the number \(\vert t\vert-\vert u\vert\) is a period of \(t\). 
\end{lemma}

\begin{lemma} (see \cite[Lemma \(8\)]{BucMichGreedy2018}) 
\label{yer789rkjfdkf}
If the palindromic lengths of prefixes (or more generally factors) of an infinite word \(x\) are bounded, then there exists a suffix \(\overline x\) of \(x\) having infinitely many palindromic prefixes.
\end{lemma}

\begin{theorem}(see \cite[Theorem of Fine and Wilf]{WILF1965})
\label{dyjf8e9ekjfdi}
Let \(t\) be a word having periods \(i\) and \(j\) with \(i\leq j\). If \(\vert t\vert\geq i+j-\gcd(i,j)\) then \(t\) has also the period \(\gcd(i,j)\). 
\end{theorem}

\section{Palindromic couples and ordinary factors}
Let
\(\PalCouple=\{(p_1,p_2)\mid p_1\in\Pal^*\mbox{ and }p_2\in\Pal^+\mbox{ and }\order(p_1p_2p_1)<2\}\mbox{}\).
We call \((p_1,p_2)\in\PalCouple\) a \emph{non-periodic palindromic couple}. 
We show several properties of non-periodic palindromes and palindromic couples.
\begin{lemma}
We have that \begin{enumerate}[ref=S\arabic*,label=S\arabic*:]
\item \label{uyue4d9bh} If \(p_1,p_2\in\Pal^+\), \(\order(p_1),\order(p_2)<2\), \(p_1\in\Prefix(p_2)\), and \(\vert p_1\vert<\vert p_2\vert\) then \(2\vert p_1\vert<\vert p_2\vert\).
\item \label{yud8u33d} If \((p_1,p_2)\in\PalCouple\), \(t\in\Prefix((p_1p_2)^{\infty})\), and \(\vert t\vert\geq {2}\vert p_1p_2\vert\) then \(\mper(t)=\vert p_1p_2\vert\).
\item \label{hhd73d22d} If \((p_1,p_2)\in\PalCouple\), \(t\in\Prefix((p_1p_2)^{\infty})\), \(p_1p_2p_1\in\Prefix(t)\), \(\xi\in\Period(t)\), and \(\vert t\vert\geq {4}\xi\)  then \(\frac{\xi}{\vert p_1p_2\vert}\in\mathbb{N}_1\).
\item \label{yuc79drx1} If \(u\in\Sigma^+\) and \(T=\{(p_1,p_2)\in\PalCouple\mid p_1p_2=u\}\) then \(\vert T\vert\leq 1\).
\item \label{dy889ekirf} If \(p\in\Pal^+\) and \(\order(p)<2\) then \((\epsilon,p)\in\PalCouple\). 
\item \label{uhe76fgt} If \((p_1,p_2),(p_3,p_4)\in\PalCouple\), \(\alpha_1,\alpha_2\in\mathbb{N}_1\), \(\alpha_1,\alpha_2\geq 2\), and \((p_1p_2)^{\alpha_1}p_1=(p_3p_4)^{\alpha_2}p_3\) then \((p_1,p_2)=(p_3,p_4)\).
\end{enumerate}
\end{lemma}
\begin{proof}
\begin{itemize}
\item \ref{uyue4d9bh}: Obvious from Lemma \ref{id8ieubmzmfj}.
    \item \ref{yud8u33d}: Let \(\xi=\mper(t)\). Clearly \(\xi\leq \vert p_1p_2\vert\).  Suppose that \(\xi<\vert p_1p_2\vert\). Then \(\vert t\vert \geq \vert p_1p_2\vert+\xi\) and thus Theorem \ref{dyjf8e9ekjfdi} implies that \(\gcd(\xi,\vert p_1p_2\vert)\in\Period(t)\). Since \(\xi<\vert p_1p_2\vert\), it is clear that  \(\gcd(\xi,\vert p_1p_2\vert)\leq \frac{1}{2}\vert p_1p_2\vert\). Since \(\vert t\vert\geq 2\vert p_1p_2\vert\) and consequently \(p_1p_2p_1\in\Factor(t)\), it follows that \(\frac{1}{2}\vert p_1p_2\vert\in\Period(p_1p_2p_1)\). This is a contradiction to \(\order(p_1p_2p_1)<2\). The property follows.
    \item \ref{hhd73d22d}: 
    Suppose that \(\vert t\vert<2\vert p_1p_2\vert\). Since \(4\xi\geq \vert t\vert\) and consequently \(\order(t)\geq 4\), it follows that \(\order(p_1p_2p_1)\geq 2\). It is a contradiction. We conclude that \(\vert t\vert\geq 2\vert p_1p_2\vert\). Then Property \ref{yud8u33d} implies that \(\vert p_1p_2\vert=\mper(t)\) and  Theorem \ref{dyjf8e9ekjfdi} implies that \(\gcd(\xi,\vert p_1p_2\vert)\in\Period(t)\). Hence \(\gcd(\xi,\vert p_1p_2\vert)=\vert p_1p_2\vert\). The property follows. 
    \item \ref{yuc79drx1}: Suppose \((p_1,p_2),(p_3,p_4)\in\PalCouple\) such that \(p_1p_2=p_3p_4\) and \(\vert p_1\vert<\vert p_3\vert\). Let \(\xi=\vert p_3p_4p_3\vert-\vert p_1p_2p_1\vert=\vert p_3\vert-\vert p_1\vert\). From Lemma \ref{id8ieubmzmfj} we have that \(\xi\in\Period(p_3p_4p_3)\). This is a contradiction since \(\order(p_3p_4p_3)<2\). The property follows.
    \item \ref{dy889ekirf}: Obvious from the definition of \(\PalCouple\).
    \item \ref{uhe76fgt}: Without loss of generality, suppose that \(\vert p_1p_2\vert\leq \vert p_3p_4\vert\). Property \ref{yud8u33d} implies that \(\vert p_1p_2\vert=\vert p_3p_4\vert=\mper((p_3p_4)^2)\). The property follows then from Property \ref{yuc79drx1}.
\end{itemize}
This ends the proof.
\end{proof}

We show that \(\ww\) has an infinite suffix with infinitely many non-periodic palindromic prefixes.
\begin{lemma}
\label{rrht7dhgfgt}
There are \(\ww_1\in\Sigma^+\) and \(\ww_2\in\Sigma^{\infty}\) such that \(\ww=\ww_1\ww_2\), \(\bb\in\Prefix(\ww_2)\), and  \(\vert\NPP(\ww_2)\vert=\infty\mbox{.}\)
\end{lemma}
\begin{proof}
Lemma \ref{yer789rkjfdkf} implies the existence of \(\ww_{0,1}\) and \(\ww_{0,2}\) such that \(\ww_0=\ww_{0,1}\ww_{0,2}\) and \(\vert\Prefix(\ww_{0,2})\cap\Pal^+\vert=\infty\).
Suppose that \(\vert\NPP(\ww_{0,2})\vert<\infty\). Then \(\vert\Prefix(\ww_{0,2})\cap\Pal^+\vert=\infty\) implies that there is \((p_1,p_2)\in \PalCouple\) such that \((p_1p_2)^j\in\Prefix(\ww_{0,2})\) for infinitely many \(j\in\mathbb{N}_1\). It follows that \(\ww_{0,2}\) has a period \(\vert p_1p_2\vert\), which is a contradiction, because we supposed \(\ww_0\) to be aperiodic and consequently \(\ww_{0,2}\) to be aperiodic. It follows that \(\vert\NPP(\ww_{0,2})\vert<\infty\). 
Let \(w_2=\pad(w_{0,2})\). It is easy to see that \(w_2\) is an infinite suffix of \(\ww\), \(\bb\in\Prefix(w_2)\), and \(\vert\NPP(w_2)\vert=\infty\).
The lemma follows. This ends the proof.
\end{proof}

We show that for every integer \(\hh_0\) there is an ordinary non-periodic palindromic factor \(\zz\) of \(\ww\) with \(\hh\geq \hh_0\) non-periodic  palindromic prefixes. 
Let \(\OrdFactor\subseteq \Sigma^+\) denote the set of all ordinary factors.
\begin{lemma}
\label{yyie9d33b}
If \(\hh_0\in\mathbb{N}_1\) then there is \(\zz\in\Factor(\ww)\cap\OrdFactor\cap\Pal^+\) such that \(\vert \NPP(\zz)\vert\geq\hh_0\), \(\order(\zz)<2\) and \(\bb\in\Prefix(\zz)\).
\end{lemma}
\begin{proof}
Given \(t\in\Sigma^+\), let  \(\tau(t)=\max\{\vert\NPP(u)\vert\mid u\in \Factor(t)\}\).
Lemma \ref{rrht7dhgfgt} asserts that there is an infinite suffix \(\ww_2\) of \(\ww\) with \(\vert\NPP(\ww_2)\vert=\infty\) and \(\bb\in\Prefix(\ww_2)\).
Then let \(v\in\Prefix(\ww_2)\) be such that \(\vert\NPP(v)\vert=\hh_0\). Clearly such \(v\) exists and  \(\bb\in\Prefix(v)\).
Let \(\overline v\in \Factor(v)\) be such that \(\vert\NPP(\overline v)\vert=\tau(v)\). 
Let \(\zz\) be the longest non-periodic palindromic prefix of \(\overline v\). Obviously \(\vert \NPP(\zz)\vert=\vert \NPP(\overline v)\vert\geq \hh_0\). From the definition of \(\tau\) it is clear that \(\zz,\overline v\in\OrdFactor\).
Suppose \(\bb\not\in\Prefix(\zz)\). Since \(\bb\in\Prefix(v)\cap\Suffix(v)\), then clearly  \(\bb\zz\bb\in\Factor(v)\) and 
\[\NPP(\bb\zz\bb)=\{\bb\}\cup\{\bb p\bb\mid p\in\NPP(\zz)\}\mbox{.}\]
It follows that \(\NPP(\bb\zz\bb)>\NPP(\zz)\), which is a contradiction. Hence \(\bb\in\Prefix(\zz)\).
The lemma follows. This ends the proof.
\end{proof}

\section{Nesting}\label{sec8903jhd}
We define formally the nested periodic structures. Let
\[\begin{split}
\widetilde\NestPer=
    \{(D,\xi)\mid D\subseteq\mathbb{N}(1,\vert\zz\vert)\mbox{ and }\xi\in\mathbb{N}_1\mbox{ and } D\not=\emptyset\mbox{ and }\\D=\Spread(D,\xi)\cap\Close(D)\mbox{ and }
    \xi\in\Period(\zz[\min(D),\max(D)])
    \}\mbox{.}
\end{split}\] 

Given \((D,\xi)\in\widetilde \NestPer\), let \(\varphi(D,\xi)=D\). Given \(M\subseteq \widetilde\NestPer\), let  \(\varphi(M)=\{\emptyset\}\cup\{D\mid (D,\xi)\in M\}\) and let \(\widetilde \varphi(M)=\bigcup_{D\in\varphi(M)}D\).
\begin{remark}
    Note that \(\emptyset\in\varphi(M)\), although \((\emptyset,\xi)\not\in\widetilde\NestPer\).
\end{remark}

Given \((D,\xi)\in\widetilde \NestPer\), let \(\Cut(D,\xi)=\{\overline D\subseteq D\mid \overline D\mbox{ is a }\xi\mbox{-cut of }D\}\mbox{.}\)

\begin{definition}
Let \(\NestPer(0)=\{ (D,\xi)\in\widetilde \NestPer\mid\vert D\vert=1 \}\mbox{.}\)
Given \(m\in\mathbb{N}_1\), let
\[\begin{split}
    \NestPer(m)=\{(D,\xi)\in\widetilde \NestPer\mid 
    \mbox{ if }\overline D\in\Cut(D,\xi)\mbox{ then there is }M\subseteq \NestPer(m-1)\\\mbox{ such that }\vert M\vert\leq \theta(m)\mbox{ and }\overline D\subseteq \widetilde \varphi(M)\subseteq \Close(\overline D)
    \}\mbox{.}
\end{split}\] 

\end{definition}

We call \((D,\xi)\in\NestPer(m)\) a \emph{nested periodic structure} \((\NestPer)\) of \emph{degree} \(m\) and we call \(D\) an  \emph{NPS cluster} of \emph{degree} \(m\). 

Given \(m\in\mathbb{N}_0\) and \(D\subseteq \mathbb{N}(1,\vert \zz\vert)\), let
\[\begin{split}\NPSCover(m,D)=\{M\subseteq \NestPer(m)\mid D\subseteq \widetilde \varphi(M)\subseteq\Close(D)
 \}\mbox{,}\end{split}\] 
  let \(\omega(m,D)=\min\{\vert M\vert\mid M\in\NPSCover(m,D)\}\), and let \(\minNPSCov(m,D)=M\in \NPSCover(m,D)\) be such that \(\vert M\vert=\omega(m,D)\). If  \(M\in\NPSCover(m,D)\) then we call \(M\) an \emph{NPS cover} of \(D\) and if  \(M=\minNPSCov(m,D)\) then we call \(M\) the \emph{minimal NPS cover} of \(D\). 
  
\begin{remark}
We have that \(\NPSCover(m,D)\not=\emptyset\), since obviously \((\{n\},1)\in\NestPer(m)\) for every \(n\in\mathbb{N}(1,\vert\zz\vert)\) and \(m\in\mathbb{N}_0\).  Consequently \(\minNPSCov(m,D)\) exists but is not uniquely determined.
\end{remark}

Given \(m,\xi\in\mathbb{N}_1\) and \(D\subseteq \mathbb{N}(1,\vert \zz\vert)\), let 
\[\begin{split}
    \omega(m,D,\xi)=\max\{\omega(m,\overline D)\mid \overline D\in\Cut(D,\xi)\}\mbox{.}
\end{split}\]

\begin{remark}
Note that if \(m\in\mathbb{N}_1\) then  
\[\begin{split}
    \NestPer(m)=\{(D,\xi)\in\widetilde \NestPer\mid \omega(m-1,D,\xi)\leq \theta(m)
    \}\mbox{.}
\end{split}\]
\end{remark}

We present several simple properties of nested periodic structures. We omit the proof. All these properties can be easily proven by induction on \(m\).
\begin{lemma}
If \(m\in\mathbb{N}_1\), \((D,\xi)\in\NestPer(m)\), \(\mu_1=\min(D)\), and \(\mu_2=\max(D)\)  then we have that: 
\begin{enumerate}[ref=R\arabic*,label=R\arabic*:]
\item \label{uj499b} If \(\mu_3\leq \mu_4\in\mathbb{N}(\mu_1,\mu_2)\), \(\overline D=D\cap\mathbb{N}(\mu_3,\mu_4)\), and \(\overline D\not=\emptyset\) 
then \((\overline D,\xi)\in\NestPer(m)\).
\item \label{euu9cb2f1}
If \(\mu_3\leq \mu_4\in\mathbb{N}(1,\vert \zz\vert)\), and \(\zz[\mu_1,\mu_2]=\zz[\mu_3,\mu_4]\), and \(\overline D=\add(D,\mu_3-\mu_1)\)
then 
 \((\overline D,\xi)\in\NestPer(m)\) and \(\vert\minNPSCov(m,D)\vert=\vert\minNPSCov(m,\overline D)\vert\).
\item \label{dy73bh2249}
If \(\mu_3\leq \mu_4\in\mathbb{N}(1,\vert \zz\vert)\), \(\zz[\mu_1,\mu_2]=(\zz[\mu_3,\mu_4])^R\), \(\mu_2\leq \mu_4\), and \(\overline D=\Mirror(\mu_1,\mu_4,D)\) then 
\((\overline D,\xi)\in\NestPer(m)\).
\item \label{yb33hfg0d}
If \(\diam(D)\leq \xi\) then \((D,\vert\zz\vert)\in\NestPer(m)\).
    \end{enumerate}    
\end{lemma}

We consider an NPS \((D,\xi)\) such that \(\zz[\min(D),\max(D)]\) is also periodic with \(\xi_2\). Then we show how to construct a set \(G\) with the diameter smaller than \(\xi_2\) and such that \(D\) is a subset of a  spread of \(G\) and \(\omega(m,G)\) is bounded.
\begin{lemma}
\label{uid93jjd} 
If \(m\in\mathbb{N}_1\), \((D,\xi)\in\NestPer(m)\), \(\mu_3\leq \mu_4\in\mathbb{N}(1,\vert \zz\vert)\), \(D\subseteq \mathbb{N}(\mu_3,\mu_4)\), \(\xi_2\in\Period(\zz[\mu_3,\mu_4])\), and \(\mu_4-\mu_3+1\leq \alpha \xi_2\) then there is \(G\subseteq \mathbb{N}(\mu_3,\mu_3+\xi_2-1)\) such that \(\omega(m,G)\leq \alpha\) and \(D\subseteq \Spread(G,\xi_2)\).
\end{lemma}
\begin{proof}
    Given \(i\in\mathbb{N}(1,\alpha)\), let \(D_i=D\cap \mathbb{N}(\mu_3+(i-1)\xi_2,\mu_3+i\xi_2-1)\)
    and let \(\overline D_i=\add(D_i, -(i-1)\xi_2)\). We have that \(D=\bigcup_{i=1}^{\alpha}D_i\).
    Let \(G=\bigcup_{i=1}^{\alpha}\overline D_i\). It is easy to see that \(G\subseteq \mathbb{N}(\mu_3,\mu_3+\xi_2-1)\) and \(D\subseteq \Spread(G,\xi_2)\).

Property \ref{uj499b} implies that \((D_i,\xi)\in \NestPer(m)\) and Property \ref{euu9cb2f1} implies that \((\overline D_i,\xi)\in\NestPer(m)\). Then obviously \(\omega(m,G)\leq \alpha\).
The lemma follows. This completes the proof.
\end{proof}

Given \(m\in\mathbb{N}_1\) and \((D,\xi)\in\NestPer(m)\), let \[\bottom(D,\xi)=\{\overline D\subseteq D\mid D=\Spread(\overline D,\xi)\cap\Close(D)\mbox{ and }\diam(\overline D)\leq \xi\}\mbox{.}\]
We call \(\overline D\in\bottom(D,\xi)\) a \emph{bottom} of \((D,\xi)\). 
For a nested periodic structure \((D,\xi)\in\NestPer(m)\), the next lemma derives an upper bound on the size of the minimal NPS cover of \(\overline D\in\Cut(D,\xi)\) as a function of the size of the minimal NPS cover of a bottom of \((D,\xi)\).
\begin{lemma}
\label{d8733j5r9}
    If \(m\in\mathbb{N}_1\), \((D,\xi)\in\NestPer(m)\), \(D_{0}\in\bottom(D,\xi)\), \(\omega(m-1,D_{0})\leq \alpha\in\mathbb{N}_1\), 
    \(\overline D\in\Cut(D,\xi)\) then \(\omega(m-1, \overline D)\leq 2\alpha\).
\end{lemma}
\begin{proof}    
The lemma is obvious for \(\overline D=\emptyset\).
Suppose that \(\overline D\not=\emptyset\). Let \(\mu_0=\min(D_0)\), \(\mu_3=\min(\overline D)\), and \(\mu_4=\max(\overline D)\).
Let \(j\in\mathbb{Z}\) be such that \(\mu_0+(j-1)\xi\leq \mu_3\leq \mu_0+j\xi-1\). 
Let \(D_1=D\cap \mathbb{N}(\mu_0+(j-1)\xi,\mu_0+j\xi-1)\)
and \(D_2=D\cap \mathbb{N}(\mu_0+j\xi,\mu_0+(j+1)\xi-1)\). Obviously \(\overline D\subseteq D_1\cup D_2\).
    Let \(M_{i}=\minNPSCov(m-1,D_i)\), where \(i\in\{1,2\}\).
    Let \[\overline M_i=\{(\overline C,\xi_2)\mid (C,\xi_2)\in M_i\mbox{ and }\overline C=D_i \cap C\}\mbox{.}\]
    Property \ref{uj499b} implies that \(\overline M_i\subseteq \NestPer(m-1)\). Property \ref{euu9cb2f1} implies that \(\vert M_i\vert=\vert \minNPSCov(D_0)\vert\leq \alpha\). It follows that \(\overline M_i\leq \alpha\), where \(i\in\{1,2\}\). 
    Let \(M_3=\overline M_1\cup \overline M_2\). Then \(\vert M_3\vert\leq 2\alpha\). It is easy to see that \(M_3\in\NPSCover(m,\overline D)\). The lemma follows.    
    This ends the proof.
\end{proof}

Given \(p_0\in\Pal^+\), let \(\PalCouple(p_0)=\{(p_1,p_2)\in \PalCouple\mid p_1p_2p_1=p_0\}\).
Given \(n\in\mathbb{N}(1,\vert \zz\vert)\), let 
\[\begin{split}\FirmPalPrefix(n)=\{ p_0\in\NPP(\zz[n,\vert \zz\vert])\mid\mbox{ for all }(p_1,p_2)\in\PalCouple(p_0)\\
\mbox{ we have that }(p_1p_2)^2p_1\not\in\Prefix(\zz[n,\vert \zz\vert])
\}\mbox{.}\end{split}\]

Given \(n\in\mathbb{N}(1,\vert \zz\vert)\), let \[\begin{split}
    \Gamma(n)=\{(p_1,p_2)\in\PalCouple\mid p_1p_2p_1\in\Prefix(\zz[n,\vert \zz\vert])\mbox{ and }\\\mbox{ if }p_1p_2p_1\in\FirmPalPrefix(n)\mbox{ then }p_1=\epsilon\}\mbox{.}
\end{split}\]

From Property \ref{uhe76fgt} and the definition of ordinary factor and \(\FirmPalPrefix\) it is straightforward to verify the following lemma. This lemma illuminates the meaning of the set \(\FirmPalPrefix\). We omit the proof.
\begin{lemma}
\label{iiubn334hd}    
If \(n\in\mathbb{N}(1,\vert \zz\vert)\) then \(\vert \Gamma(n)\vert\leq \hh\).
\end{lemma}

Given \(n\in \mathbb{N}(1,\vert \zz\vert)\), let \[\begin{split}\PalExt_1(n)=\{(n,p_1,p_2,\alpha) \mid (p_1p_2)^{\alpha}p_1\in\Prefix(\zz[n,\vert \zz\vert])\mbox{ and }(p_1,p_2)\in\PalCouple\\\mbox{ and }\alpha\in\mathbb{N}_1\mbox{ and }p_1p_2p_1\not\in\FirmPalPrefix(n)\}\mbox{,}\end{split}\]
\[\begin{split}\PalExt_2(n)=\{(n,\epsilon,p_1,1)\mid (\epsilon,p_1)\in\PalCouple\mbox{ and }p_1\in\FirmPalPrefix(n)\}\mbox{,}\end{split}\] 
and let \(\PalExt(n)=\PalExt_1(n)\cup \PalExt_2(n)\). 
We call \((n,p_1,p_2,\alpha)\in\PalExt(n)\) a \emph{palindromic extension tuple} of the position \(n\).

Given \(D\subseteq \mathbb{N}(1,\vert \zz\vert)\), let \(\PalExt(D)=\bigcup_{n\in D}\PalExt(n)\).
Given \(n\in\mathbb{N}(1,\vert\zz\vert)\) and \((n,p_1,p_2,\alpha)\in\PalExt(n)\), let
\(\sigma(n,p_1,p_2,\alpha)=n-1+\vert (p_1p_2)^{\alpha}p_1\vert\mbox{}\)
and let \[\sigma(\PalExt(D))=\{\sigma(n,p_1,p_2,\alpha)\mid (n,p_1,p_2,\alpha)\in\PalExt(D)\}\mbox{.}\]
We call \(\sigma(n,p_1,p_2,\alpha)\in\mathbb{N}(1,\vert \zz\vert)\) a \emph{palindromic extension} of \(n\).

\begin{remark}
From Property \ref{uhe76fgt} and the definitions of \(\FirmPalPrefix\) and \(\PalExt\) it is easy to see that \(\vert \PalExt(n)\vert=\vert\sigma(\PalExt(n))\vert\) and \(\PalExt_1(n)\cap\PalExt_2(n)=\emptyset\).  Also it is obvious that if \(p_0\in\Prefix(\zz[n,\vert \zz\vert])\cap\Pal^+\) then there is \((n,p_1,p_2,\alpha)\in\PalExt(n)\) such that \(p_0=(p_1p_2)^{\alpha}p_1\).

It is clear that if \(p_0\in\FirmPalPrefix(\zz, n)\), then \(\order(p_0)<2\). Hence from Property \ref{dy889ekirf} it follows that if \(p_0\in\FirmPalPrefix(n)\) then \((\epsilon,p_0)\in\PalCouple\), hence the definition of \(\PalExt_2(n)\) makes sense.
\end{remark}

Without a proof we state the following simple observation.
\begin{lemma}
\label{uud8kt87e}    
If \((n,p_1,p_2,\alpha)\in\PalExt(n)\) and \(\overline \alpha\in\mathbb{N}(1,\alpha)\) then \((n,p_1,p_2,\overline \alpha)\in\PalExt(n)\).
\end{lemma}

We prove that all palindromic extensions of a position \(n\) are a subset of the union of \(\hh\) NPS clusters of degree \(1\).
\begin{lemma}
\label{iodi9efjg}
    If \(n\in\mathbb{N}(1,\vert \zz\vert)\) and  \(D=\sigma(\PalExt(n))\) then \(\omega(1,D)\leq \hh\).
\end{lemma}
\begin{proof}
Let \(K=\{(p_1,p_2)\mid (n,p_1,p_2,\alpha)\in\PalExt(n)\}\mbox{.}\) 
    Suppose \((p_1,p_2)\in K\). Let \(D=\{n+\vert(p_1p_2)^{\alpha}p_1\vert-1\mid (n,p_1,p_2,\alpha)\in \PalExt(n)\}\mbox{}\) and \(\xi=\vert p_1p_2\vert\). 
    Let \(\mu_1=n+\vert p_1p_2p_1\vert-1\) and let \(\mu_2=\max(D)\). Lemma \ref{uud8kt87e} implies that \(\mu_1\in D\) and in consequence \(\mu_1=\min(D)\). Note that \(\mu_1\in\varphi(\NestPer(0))\). Lemma \ref{uud8kt87e} implies also that \(D=\Spread(\{\mu_1\},\xi)\cap\mathbb{N}(\mu_1,\mu_2)\). Obviously \(\omega(0,D,\xi)\leq 1\).
    
    Then it is easy to see that \((D,\xi) \in\NestPer(1)\mbox{.}\)
    From Lemma \ref{iiubn334hd} we have that \(\vert K\vert\leq \hh\). The lemma follows. 
This ends the proof.
\end{proof}
    
Let 
\[\begin{split}
    \Omega=\{m\in\mathbb{N}_1\mid \mbox{ for every }(D,\xi)\in\NestPer(m-1)\mbox{ we have that }\\\omega(m,\widehat D)\leq c_3\hh\mbox{, where }\widehat D=\sigma(\PalExt(D))\}\mbox{.}
\end{split}\]
\begin{remark}
Lemma \ref{iodi9efjg} implies that \(1\in\Omega\). 
The main result of Section \ref{sec8903jhd} will show that \(\Omega=\mathbb{N}_1\); see Corollary \ref{yysbn44df}
\end{remark}

\subsection{Inside NPS palindromic extension}
Given \(n_1,\xi\in\mathbb{N}(1,\vert\zz\vert)\), let \(\perProlong(n_1,\xi)=n_2\in\mathbb{N}(1,\vert\zz\vert)\) be such that \(\xi\in\Period(\zz[n_1,n_2])\) and if \(n_2< \vert\zz\vert\) then \(\xi\not\in\Period(\zz[n_1,n_2+1])\). We call \(\zz[n_1,n_2]\) the \emph{periodic prolongation} of \((n_1,\xi)\).

For the nested palindromic structure \((D,\xi)\in\NestPer(m)\), we show that the palindromic extensions of \(D\) inside of the periodic prolongation of \((\min(D),\xi)\) are a subset of a spread of palindromic extensions of a bottom \(C_0\) of \((D,\xi)\). 
This result will allow us to restrict our attention only to a bottom of \(D\) when identifying the palindromic extensions of \(D\) inside of a periodic prolongation.

\begin{lemma}
    \label{ppdio08bg}
    If \(m\in \mathbb{N}_1\), \((D,\xi)\in\NestPer(m)\), \(\mu_1=\min(D)\), 
    \(C_0=D\cap \mathbb{N}(\mu_1,\mu_1+\xi-1)\),
    \(\mu_3=\perProlong(\mu_1,\xi)\),
    \(\widehat D=\sigma(\PalExt(D))\cap \mathbb{N}(\mu_1,\mu_3)\), 
    and
     \[\widehat C_0=\sigma(\PalExt(C_0))\cap \mathbb{N}(\mu_1,\mu_3)\]
     then 
    \(\widehat D\subseteq \Spread(\widehat C_0,\xi)\).    
\end{lemma}
\begin{proof}
    Suppose that there is \(n\in D\) and \((n,p_1,p_2,\alpha)\in\PalExt(n)\) such that \(\sigma(n,p_1,p_2,\alpha)\in\mathbb{N}(\mu_1,\mu_3)\) and \(\sigma(n,p_1,p_2,\alpha)\not\in \Spread(\widehat C_0,\xi)\). 
Let \(n_0\in C_0\) be such that \begin{equation}\label{hy8dj33f}n\equiv n_0\pmod{\xi}\mbox{.}\end{equation}
Realize that \(C_0\in\bottom(D,\xi)\).  Hence \(n_0\) exists and is unique. From the definition of \(C_0\) it is clear that \(n_0\leq n\).
Then obviously \((n_0,p_1,p_2,\alpha)\in\PalExt(n_0)\) and \(\sigma(n_0,p_1,p_2,\alpha)\in\mathbb{N}(\mu_1,\mu_3)\), since \(\xi\in\Period(\zz[\mu_1,\mu_3])\). Let \[\widehat n_0=\sigma(n_0,p_1,p_2,\alpha)=n_0-1+\vert(p_1p_2)^{\alpha}p_1\vert\in\widehat C_0\] and \[\widehat n=\sigma(n,p_1,p_2,\alpha)=n-1+\vert(p_1p_2)^{\alpha}p_1\vert\in\widehat D\mbox{.}\] From (\ref{hy8dj33f}) it follows that \(\widehat n\in \Spread(\{\widehat n_0\},\xi)\), which is a contradiction to our assumption. We conclude that \(\widehat D\subseteq \Spread(\widehat C_0,\xi)\). 
This ends the proof.
\end{proof}

We show that all palindromic extensions \(\widehat C_0\) of the bottom \(C_0\) of NPS \((D,\xi)\in\NestPer(m)\) inside of the periodic prolongation of \((\min(D),\xi)\) are a subset of the spread of a subset \(G\) of \(\widehat C_0\) such that \(G\) has a ``short'' diameter (\(\leq c_1\xi\)). 
\begin{lemma}
    \label{uuidje933h}
        If \(m\in \mathbb{N}_1\), \((D,\xi)\in\NestPer(m)\), \(\mu_1=\min(D)\), \(\mu_3=\perProlong(\mu_1,\xi)\), 
    \(C_0=D\cap \mathbb{N}(\mu_1,\mu_1+\xi-1)\),    
     \(\widehat C_0=\sigma(\PalExt(C_0))\cap \mathbb{N}(\mu_1,\mu_3)\),
     and \(G=\widehat C_0\cap \mathbb{N}(\mu_1,\mu_1+c_1\xi-1)\),
     then 
    \(\widehat C_0\subseteq \Spread(G,\xi)\).    
\end{lemma}
\begin{proof}
Suppose  \((n,p_1,p_2,\alpha)\in \PalExt(C_0)\) and \(\widehat n=\sigma(n,p_1,p_2,\alpha)\in\widehat C_0\setminus G\). Since \(\max(C_0)\leq \mu_1+\xi-1\) and \(\widehat n> \mu_1+c_1\xi-1\) it is clear that \(\widehat n-n> (c_1-1)\xi\). 
It follows that \(\frac{\xi}{\vert p_1p_2\vert}=\delta \in\mathbb{N}_1\), because \(\xi\in\Period(\zz[n,\widehat n])\); see Property \ref{hhd73d22d}. This implies that there is \(\alpha_0\in\mathbb{N}(1,\delta)\) such that \(\alpha \vert p_1p_2\vert-\alpha_0\vert p_1p_2\vert\equiv 0\pmod{\xi}\). 
Then clearly \((n,p_1,p_2,\alpha_0)\in\PalExt(C_0)\), \(\sigma(n,p_1,p_2,\alpha_0)\in G\), and \(\widehat n-\sigma(n,p_1,p_2,\alpha_0)\equiv 0\pmod{\xi}\). 
It follows that  \(\widehat n\in\Spread(G,\xi)\).
This completes the proof.
\end{proof}

The main result of this subsection says that the palindromic extensions of an NPS cluster of degree \(m\in\Omega\) inside of the periodic prolongation are a subset of an NPS cluster of degree \(m+1\).
\begin{theorem}
    \label{f8u556ju}    
     If \(m\in \Omega\), \((D,\xi)\in\NestPer(m)\), \(\mu_1=\min(D)\), \(\mu_3=\perProlong(\mu_1,\xi)\),  and
	\(\widehat D=\sigma(\PalExt(D)) \cap \mathbb{N}(\mu_1,\mu_3)\)
     then 
    there is \(G\subseteq \mathbb{N}(\mu_1,\mu_3)\) such that \(\widehat D\subseteq G\) and 
     \((G,\xi)\in\NestPer(m+1)\).
\end{theorem}
\begin{proof}
Let \(C_0=D\cap\mathbb{N}(\mu_1,\mu_1+\xi-1)\) and let \(\widehat C_0=\sigma(\PalExt(C_0))\cap\mathbb{N}(\mu_1,\mu_3)\).
From the definition of \(\NestPer(m)\) we have that \(\omega(m-1,C_0)\leq \theta(m)\). 
It follows that there are \(g\leq \theta(m)\) and \(C_i\in\varphi(\NestPer(m-1))\) such that \(C_0\subseteq\bigcup_{i=1}^gC_i\) and \(C_i\subseteq\mathbb{N}(\mu_1,\mu_1+\xi-1)\), where \(i\in\mathbb{N}(1,g)\).

Let \[\widehat C_i=\sigma(\PalExt(C_i))\cap\mathbb{N}(\mu_1,\min\{\mu_3,\mu_1+c_1\xi-1\})\mbox{, where  }i\in\mathbb{N}(1,{g})\mbox{.}\]
Lemma \ref{uuidje933h} implies that \begin{equation}\label{ndj88fdhu} \widehat C_0\subseteq \bigcup_{i=1}^g\Spread(\widehat C_i,\xi)\mbox{.}\end{equation}


Because \(m\in\Omega\), we have that \(\omega(m,\sigma(\PalExt(C_i)))\leq c_3\hh\). Then Property \ref{uj499b}  implies that \(\omega(m,\widehat C_i)\leq c_3\hh\).
  It follows that there are \(\widehat C_{i,j}\in\varphi(\NestPer(m))\) such that \(j\in\mathbb{N}(1,c_3\hh)\),  \(\widehat C_{i,j}\subseteq\mathbb{N}(\mu_1,\mu_1+c_1\xi-1)\) and \begin{equation}\label{cnbgfbv65}
      \widehat C_i\subseteq\bigcup_{j=1}^{ c_3\hh}\widehat C_{i,j}\mbox{.}
  \end{equation}

Let \(\widehat C_{i,j,\delta}\subseteq \mathbb{N}(\mu_1,\mu_1+\xi-1)\) be such that \(\delta\in\mathbb{N}(1,c_1)\), 
\(\widehat C_{i,j,\delta}\in\varphi(\NestPer(m))\), and \begin{equation}\label{ood93jfb3}\widehat C_{i,j}\subseteq\bigcup_{\delta=1}^{c_1}\Spread(\widehat C_{i,j,\delta},\xi)\mbox{.}\end{equation} Lemma \ref{uid93jjd} asserts that such \(\widehat C_{i,j,\delta}\) exist.
Let \(G_0=\bigcup_{i,j,\delta}\widehat C_{i,j,\delta}\). It is easy to see that \(G_0\subseteq \mathbb{N}(\mu_1,\mu_1+\xi-1)\). Since \(i\leq\theta(m)\), \(j\leq c_3\hh\), and \(\delta\leq c_1\), we have that \begin{equation}\label{iid8bdh2}\omega(m,G_0)\leq c_1c_3\hh\theta(m)\mbox{.}\end{equation}
Let \(G=\Spread(G_0,\xi)\cap\mathbb{N}(\mu_1,\mu_3)\). Lemma \ref{d8733j5r9} and (\ref{iid8bdh2}) imply that 
\begin{equation}
    \label{hh8cb2gfb3}
\omega(m,G,\xi)\leq 2c_1c_3\hh\theta(m)\leq\theta(m+1)\mbox{.}\end{equation}

From (\ref{ndj88fdhu}), (\ref{cnbgfbv65}), and (\ref{ood93jfb3}) we have that 
    \(\widehat C_0\subseteq \Spread(G_0,\xi)\) and consequently \(\Spread(\widehat C_0,\xi)\subseteq \Spread(G_0,\xi)\). 
Lemma \ref{ppdio08bg} implies that \(\widehat D\subseteq\Spread(\widehat C_0,\xi)\). It follows that \(\widehat D\subseteq G\). 
 Then from (\ref{hh8cb2gfb3}) we conclude that \((G,\xi)\in\NestPer(m+1)\). This ends the proof.
\end{proof}

\subsection{Runs and canonical palindromes}
In this subsection we define runs and we derive an upper bound for the number of distinct runs covering a given position of the factor \(\zz\). 
 Let \[\begin{split}\Run=\{(\nu_1,\nu_2,\xi)\mid \nu_1\leq\nu_2\in \mathbb{N}(1,\vert \zz\vert)\mbox{ and }\xi\in\Period(\zz[\nu_1,\nu_2])\mbox{ and }\\\nu_2-\nu_1+1\geq\xi\mbox{ and }\nu_2<\vert \zz\vert\implies\xi\not\in\Period(\zz[\nu_1,\nu_2+1])\\\mbox{ and }\nu_1>1\implies \xi\not\in\Period(\zz[\nu_1-1,\nu_2]\mbox{ and }\\\mbox{ there is }(p_1,p_2)\in\PalCouple\mbox{ such that }\vert p_1p_2\vert=\xi\mbox{ and }p_1p_2\in\Prefix(\zz[\nu_1,\nu_2])\}\mbox{.}\end{split}\] 

We call \((\nu_1,\nu_2,\xi)\in \Run\) a \emph{run}. 
The term ``run'' was already used in \cite{FrPuZa}. Our definition has the additional requirement that \(p_1p_2\in\Prefix(\zz[\nu_1,\nu_2])\) and consequently \(\zz[\nu_1,\nu_2]\in\Prefix((p_1p_2)^{\infty})\).

Given \((p_1,p_2)\in \PalCouple\), \(n_1\in\mathbb{N}(1,\vert \zz\vert)\) with \(p_1p_2\in\Prefix(\zz[n_1,\vert \zz\vert])\), 
let \(\pi(n_1,p_1,p_2)=\{\gamma_1,\gamma_2\}\subset\mathbb{Q}\) be such that \(n_3=n_1+\vert p_1p_2\vert-1\), \(n_2=n_3-\vert p_2\vert+1\), \[\gamma_1=\frac{n_1+n_2-1}{2}\in\mathbb{Q}\mbox{,}\quad\mbox{ and }\quad\gamma_2=\frac{n_2+n_3}{2}\in\mathbb{Q}\mbox{.}\]  
\begin{remark}
If \(\gamma\in\pi(n_1,p_1,p_2)\) then \(\gamma\in\mathbb{Q}\) is an integer or half-integer and \(\gamma\) is the center position of the palindrome \(p_1\) or \(p_2\). Note that if \(p_1=\epsilon\) then \(n_2=n_1\) and \(\gamma_1=n_1-\frac{1}{2}\).
\end{remark}

\begin{lemma}
    \label{ff7e6erd}
    If \(n_1\in\mathbb{N}(1,\vert \zz\vert)\), \((p_1,p_2)\in\PalCouple\), \(p_1p_2\in\Prefix(\zz[n_1,\vert \zz\vert])\), \(\xi=\vert p_1p_2\vert\) , and \(\{\gamma_1,\gamma_2\}=\pi(n_1,p_1,p_2)\) then \(\vert \gamma_1-\gamma_2\vert=\frac{\xi}{2}\in\mathbb{Q}\).
\end{lemma}
\begin{proof} Let \(n_3=n_1+\xi-1\) and \(n_2=n_3-\vert p_2\vert+1\).
    From the definition of \(\pi\) we have that
    \[\begin{split}
        \gamma_1-\gamma_2=\frac{n_1+n_2-1}{2}-\frac{n_2+n_3}{2}=\frac{n_1-1-n_3}{2}=\\\frac{n_1-1-(n_1+\xi-1)}{2}=-\frac{\xi}{2}\mbox{.}
    \end{split}\]
    The lemma follows. This ends the proof.
\end{proof}

Given \((\nu_1,\nu_2,\xi)\in\Run\), let \(\PalCouple(\nu_1,\xi)=(p_1,p_2)\in\PalCouple\) be such that 
 \(p_1p_2\in\Prefix(\zz[\nu_1,\vert \zz\vert])\) and \(\xi=\vert p_1p_2\vert\). The definition of a run asserts the existence of such \((p_1,p_2)\) and Property \ref{yuc79drx1} asserts that \((p_1,p_2)\) is unique.
 
Given \((\nu_1,\nu_2,\xi)\in\Run\), let
\[\begin{split}
    \RunPal(\nu_1,\nu_2,\xi)=\{(n_1,n_2)\in\mathbb{N}^2_1(\nu_1,\nu_2)\mid  (p_1,p_2)=\PalCouple(\nu_1,\xi)\mbox{ and } \\\frac{n_1+n_2}{2}\in\Spread(\pi(\nu_1,p_1,p_2),\xi)\}\mbox{.}
\end{split}\]

\begin{remark}
    Realize that if \((n_1,n_2)\in\RunPal(\nu_1,\nu_2,\xi)\) then \(\zz[n_1,n_2]\in \Pal^+\). 
\end{remark}
We call \(\zz[n_1,n_2]\) a \emph{canonical palindrome} of the run \((\nu_1,\nu_2,\xi)\in\Run\), where \((n_1,n_2)\in\RunPal(\nu_1,\nu_2,\xi)\).

\begin{lemma}
    \label{hhyd89j33b}
    If \((n_1,n_2),(n_1,n_3)\in\RunPal(\nu_1,\nu_2,\xi)\) then \(n_3-n_2\equiv 0\pmod{\xi}\).
\end{lemma}
\begin{proof}
From Lemma \ref{ff7e6erd} it follows that there is \(\alpha\in\mathbb{Z}\) such that \[\frac{n_2+n_1}{2}-\frac{n_3+n_1}{2}=\frac{n_2-n_3}{2}=\alpha\frac{\xi}{2}\mbox{.}\]
The lemma follows.
\end{proof}

Given \(n\in \mathbb{N}(1,\vert \zz\vert)\), let
\[\begin{split}
    \CovPalAll(n)=&\{(n_1,n_2)\in \mathbb{N}^2(1,\vert \zz\vert)\mid n\in\mathbb{N}(n_1,n_2)\mbox{ and }\zz[n_1,n_2]\in\Pal^+\}\mbox{,}\\
\CovPalAllLeft(n)=&\{(n_1,n_2)\in \CovPalAll(n)\mid \frac{n_2+n_1}{2}\geq n\}\mbox{,}\\
\CovPalAllRight(n)=&\{(n_1,n_2)\in \CovPalAll(n)\mid \frac{n_2+n_1}{2}\leq n\}\mbox{,}\\
\CovPalEdge(n)=&\{(n_1,n_2)\in\CovPalAll(n)\mid \mbox{ if }n_1>1\mbox{ and }n_2<\vert \zz\vert\mbox{ then }\\ &\zz[n_1-1,n_2+1]\not\in\Pal^+\}\mbox{,}\\
\CovPalEdgeLeft(n)=&\CovPalAllLeft(n)\cap\CovPalEdge(n)\mbox{, and }\\
\CovPalEdgeRight(n)=&\CovPalAllRight(n)\cap\CovPalEdge(n)\mbox{.}
\end{split}\]

We call \((n_1,n_2)\in\CovPalAll(n)\) a \emph{covering palindrome} of \(n\) and we call \((n_1,n_2)\in\CovPalEdge(n)\) an \emph{edge covering palindrome} of \(n\).

Given \((n_1,n_2)\in\mathbb{N}(1,\vert\zz\vert)\) with \(\zz[n_1,n_2]\in\Pal^+\), let \(\toPalCouple(n_1,n_2)= (p_1,p_2)\in\PalCouple\) be such that:
\begin{itemize}
    \item If \(\zz[n_1,n_2]\in\FirmPalPrefix(n_1)\) then \((p_1,p_2)=(\epsilon,\zz[n_1,n_2])\).
    \item If \(\zz[n_1,n_2]\not\in\FirmPalPrefix(n_1)\) then \((p_1p_2)^{2}p_1\in\Prefix(\zz[n_1,\vert \zz\vert])\), \(\zz[n_1,n_2]=(p_1p_2)^{\alpha}p_1\), and \(\alpha\in\mathbb{N}_1\).
\end{itemize}
\begin{remark}
    Note that \(\toPalCouple(n_1,n_2)\) exists and is unique, see Lemma \ref{tudjkdi8545} and Property \ref{uhe76fgt}. Property \ref{dy889ekirf} asserts that \((\epsilon,\zz[n_1,n_2])\in\PalCouple\).
\end{remark}

Given \(n\in\mathbb{N}(1,\vert \zz\vert)\),  and \((n_1,n_2)\in\CovPalEdgeLeft(n)\), let \[\ctPC(n_1,n_2,n)=\toPalCouple(n,\Mirror(n_1,n_2,n))\in\PalCouple\mbox{.}\]
We call \(\ctPC(n_1,n_2,n)\) a \emph{center palindromic couple} of \((n_1,n_2)\).
Given \(n\in\mathbb{N}(1,\vert \zz\vert)\), \((p_1,p_2)\in\PalCouple\), let
\[\begin{split}\CovPalEdgeLeft(n,p_1,p_2)=&\{(n_1,n_2)\in\CovPalEdgeLeft(n)\mid (p_1,p_2)=\ctPC(n_1,n_2,n)\}\mbox{,}\\
\CovPalCmd(n,p_1,p_2)=&\{(n_1,n_2)\in\CovPalEdgeLeft(n,p_1,p_2)\mid \zz[n_1,n_2]\not\in\Factor((p_1p_2)^{\infty})\}\mbox{, }\\\CovPalCmd(n)=&\bigcup_{(p_1,p_2)\in\PalCouple}\CovPalCmd(n,p_1,p_2)\mbox{, and }\\\ctPC(n)=&\{(p_1,p_2)\in\PalCouple\mid \CovPalEdgeLeft(n,p_1,p_2)\not=\emptyset\}\mbox{.}\end{split}\]

If \((n_1,n_2)\in\CovPalCmd(n)\), then we call \(\zz[n_1,n_2]\) a \emph{compound covering palindrome}. 

\begin{remark}
Note that \(\CovPalCmd(n)\) is the set all of edge covering palindromes from \(\CovPalEdgeLeft(n)\), that are not a power of their center palindromic couple.
\end{remark}

The next lemma presents an upper bound for the number of compound covering palindromes covering a position \(n\).
\begin{lemma}
\label{u7m2m7ft8}
If \(n\in\mathbb{N}(1,\vert \zz\vert)\) then \(\vert \CovPalCmd(n)\vert\leq \hh\).
\end{lemma}
\begin{proof}
Let \((n_1,n_2)\in\CovPalCmd(n)\) and  \((p_1,p_2)=\ctPC(n_1,n_2,n)\mbox{.}\) 

Since \(r[n_1,n_2]\not\in\Factor((p_1p_2)^{\infty})\), 
it is easy to verify that there are unique \(a\in\Sigma\), \(\alpha\in\mathbb{N}_1\), and \(u,v\in\Sigma^*\) such that 
\(r[n_1,n_2]=uav(p_1p_2)^{\alpha}p_1v^Rau^R\), \(\overline n=\Mirror(n_1,n_2,n)\), \((p_1p_2)^{\alpha}p_1=r[n,\overline n]\), 
\(v(p_1p_2)^{\alpha}p_1v^R\in\Factor((p_1p_2)^{\infty})\), and  \(av(p_1p_2)^{\alpha}p_1v^Ra\not\in\Factor((p_1p_2)^{\infty})\).
We conclude that \(\vert \CovPalCmd(n,p_1,p_2)\vert\leq 1\).

Lemma \ref{iiubn334hd} implies that \(\vert \ctPC(n)\vert\leq \hh\) and consequently \(\vert \CovPalCmd(n)\vert\leq \hh\mbox{.}\)
The lemma follows. This ends the proof.
\end{proof}

Given \(n_1\leq n_2\in\mathbb{N}(1,\vert\zz\vert)\) with \(\zz[n_1,n_2]\in\Pal^+\), let 
    \(\pToRun(n_1,n_2)=(\nu_1,\nu_2,\xi)\in\Run\mbox{}\) be such that \(n_1,n_2\in\mathbb{N}(\nu_1,\nu_2)\), \(\xi=\vert p_1p_2\vert\), and \((p_1,p_2)=\toPalCouple(n_1,n_2)\). Obviously \(\pToRun(n_1,n_2)\) exists and is unique.
        
The next propositions present an upper bound for the number of runs ``constructed'' from edge covering palindromes covering a position \(n\).
\begin{proposition}
\label{yue7cv39j}
    If \(n\in\mathbb{N}(1,\vert r\vert)\) then     
    \(\vert \pToRun(\CovPalEdgeLeft(n))\vert\leq {2}\hh\mbox{.}\)    
\end{proposition}
\begin{proof}
Lemma \ref{u7m2m7ft8} says that \(\vert\CovPalCmd(n)\vert\leq \hh\) and hence this implies that \begin{equation}\label{hys778evv3v3}\vert \pToRun(\CovPalCmd(n))\vert\leq\hh\mbox{.}\end{equation}
Let \(\CovPalEdgeLeft_0(n)=\CovPalEdgeLeft(n)\setminus \CovPalCmd(n)\).

Suppose \((p_1,p_2)\in \PalCouple\). Let \(K(n,p_1,p_2)=\CovPalEdgeLeft_0(n)\cap \CovPalEdgeLeft(n,p_1,p_2)\) and \(\overline K(n,p_1,p_2)=\pToRun(K(n,p_1,p_2))\mbox{.}\)

Suppose \((n_1,n_2), (\overline n_1,\overline n_2)\in K(n,p_1,p_2)\mbox{.}\) Realize that \((n_1,n_2), (\overline n_1,\overline n_2)\in\CovPalEdgeLeft_0(n)\). It follows that \(\zz[n_1,n_2],\zz[\overline n_1,\overline n_2]\in\Factor((p_1p_2)^{\infty})\) and \(p_1p_2p_1\in\Prefix(\zz[n,\vert \zz\vert])\). Thus it is easy to verify that  \(\pToRun(n_1,n_2)=\pToRun(\overline n_1,\overline n_2)\mbox{.}\) In consequence \(\vert K(n,p_1,p_2)\vert\leq 1\) and \(\vert\overline K(n,p_1,p_2)\vert\leq 1\).
Lemma \ref{iiubn334hd} implies that \(\vert \ctPC(n)\vert\leq \hh\), hence we have that 
\(\vert \pToRun(\CovPalEdgeLeft_0(n))\vert\leq \hh\mbox{.}\)
The proposition follows then from (\ref{hys778evv3v3}); realize that \(\CovPalEdgeLeft(n)=\CovPalCmd(n)\cup\CovPalEdgeLeft_0(n)\mbox{.}\)
This ends the proof.
\end{proof}

The main result of this subsection says that the number of runs constructed from edge covering palindromes covering the position \(n\) is bounded by \({4}\hh\).
\begin{corollary}
    \label{jud9ujorr}
    If \(n\in\mathbb{N}(1,\vert r\vert)\) then     
    \(\vert \pToRun(\CovPalEdge(n))\vert\leq {4}\hh\mbox{.}\)  
\end{corollary}
\begin{proof}
    Let \(\overline n=\vert \zz\vert-n\). Since \(\zz\in\Pal^+\), it is clear that \(\CovPalEdgeRight(n)=\CovPalEdgeLeft(\overline n)\).
    From Proposition \ref{yue7cv39j} we have that \(\vert \pToRun(\CovPalEdgeLeft(\overline n))\vert\leq {2}\hh\mbox{,}\) hence it follows that \(\vert \pToRun(\CovPalEdgeRight(n))\vert\leq {2}\hh\mbox{.}\) Since \(\CovPalEdge(n)=\CovPalEdgeLeft(n)\cup\CovPalEdgeRight(n)\) the corollary follows.    
    This completes the proof.
\end{proof}

We define two special subsets of nested periodic structures. Given \(m\in\mathbb{N}_1\), 
let \[\begin{split}\NestPer_1(m)=&\{(D,\xi)\in\NestPer(m)\mid D\cap\mathbb{N}(\mu_3-c_1\xi+1,\mu_3)=\emptyset \\&\mbox{ and }\mu_3=\perProlong(\min(D),\xi)\}\mbox{, and }\\\NestPer_2(m)=&\{(D,\xi)\in\NestPer(m)\mid D\subseteq\mathbb{N}(\mu_3-c_1\xi+1,\mu_3)\\&\mbox{ and }\mu_3=\perProlong(\min(D),\xi)\}\mbox{.}\end{split}\]

Given \((D,\xi)\in\NestPer(m)\), let \(\separate(D,\xi)=(D_1,D_2)\) be such that 
\(D_2=D\cap \mathbb{N}(\mu_3-c_1\xi+1,\mu_3)\) and \(D_1=D\setminus D_2\), where \(\mu_3=\perProlong(\min(D),\xi)\). We call \((D_1,D_2)\) a \emph{separation} of the nested periodic structure \((D,\xi)\). Note that if \(D_1\not=\emptyset\) then \((D_1,\xi)\in\NestPer_1(m)\) and if \(D_2\not=\emptyset\) then \((D_2,\xi)\in\NestPer_2(m)\); see Property \ref{uj499b}. The purpose of the subsets \((D_1,D_2)\) is to separate the positions the are ``far'' and ``close'' to the periodic prolongation \(\mu_3\). These far and close positions will be manipulated differently.

We show that the palindromic extensions of an NPS cluster \(D_2\) of degree \(m\) can be covered by \(c_1\) NPS clusters of degree \(m+1\).
\begin{lemma}
    \label{hhbd628d}
    If \(m\in\Omega\), \((D_2,\xi)\in\NestPer_2(m)\), and \(\widehat D_2=\sigma(\PalExt(D_2))\) then 
    \(\omega(m+1,\widehat D_2)\leq c_1\).
\end{lemma}
\begin{proof}
From the definition of \(\NestPer_2(m)\), we have that \(\diam(D_2)\leq c_1\xi\). Let \(D_{2,i}=D_2\cap\mathbb{N}(\min(D_2)+(i-1)\xi,\min(D_2)+i\xi-1)\), where \(i\in\mathbb{N}(1,c_1)\). We have that \(D_2=\bigcup_{i=1}^{c_1}D_{2,i}\). If \(D_{2,i}\not=\emptyset\) then Property \ref{uj499b} implies that \((D_{2,i},\xi)\in\NestPer(m)\) and Property \ref{yb33hfg0d} implies \((D_{2,i},\vert \zz\vert)\in\NestPer(m)\). 
Let \(\widehat D_{2,i}=\sigma(\PalExt(D_{2,i}))\). 
We have that \(\perProlong(\min(D_{2,i}),\vert \zz\vert)=\vert \zz\vert\) and thus Theorem \ref{f8u556ju} implies that if \(\widehat D_{2,i}\not=\emptyset\) then \((\widehat D_{2,i},\vert \zz\vert)\in\NestPer(m+1)\) for all \(i\in\mathbb{N}(1,c_1)\).
The lemma follows. This ends the proof.
\end{proof}

Given \(m\in\mathbb{N}_1\), \((D_1,\xi_1)\in\NestPer_1(m)\), and \((\nu_1,\nu_2,\xi_2)\in\Run\), let \[\begin{split}\PalExt(D_1,\xi_1,\nu_1,\nu_2,\xi_2)=\{(n,p_1,p_2,\alpha)\in\PalExt(D)\mid \mu_3=\perProlong(D_1,\xi_1)\\\mbox{ and }\mu_3<\sigma(n,p_1,p_2,\alpha) \mbox{ and }(n,\sigma(n,p_1,p_2,\alpha))\in\RunPal(\nu_1,\nu_2,\xi_2)\}\mbox{.}\end{split}\]

Let \(\Psi\) be the set of all \({6}\)-tuples \((m,D_1,\xi_1,\nu_1,\nu_2,\xi_2)\) such that 
 \(m\in\mathbb{N}_1\), \((D_1,\xi_1)\in\NestPer_1(m)\), \((\nu_1,\nu_2,\xi_2)\in\Run\), and \(\PalExt(D_1,\xi_1,\nu_1,\nu_2,\xi_2)\not=\emptyset\). 

Given \((m,D_1,\xi_1,\nu_1,\nu_2,\xi_2)\in\Psi\), let \[\minRunPal(D_1,\xi_1,\nu_1,\nu_2,\xi_2)=(d_1,d_2)\in\CovPalAll(\mu_3+1)\] be such that \(\mu_3=\perProlong(D_1,\xi_1)\), 
 \[\begin{split}d_1=&\min\{n \mid (n,p_1,p_2,\alpha)\in\PalExt(D_1,\xi_1,\nu_1,\nu_2,\xi_2)\}\mbox{, and }\\ d_2=&\min\{\sigma(d_1,p_1,p_2,\alpha)\mid (d_1,p_1,p_2,\alpha)\in\PalExt(D_1,\xi_1,\nu_1,\nu_2,\xi_2))\}\mbox{.}\end{split}\]
Clearly \((d_1,d_2)\in\RunPal(\nu_1,\nu_2,\xi_2)\). We call \((d_1,d_2)\) the \emph{minimal canonical palindrome} of \((D_1,\xi_1,\mu_1,\mu_2,\xi_2)\).

\subsection{Outside NPS palindromic extension}
For this subsection suppose 
\begin{itemize}
    \item 
\((m,D_1,\xi_1,\nu_1,\nu_2,\xi_2)\in\Psi\), 
\item \(\mu_3=\perProlong(\min(D_1),\xi_1)\), 
\item\((d_1,d_2)=\minRunPal(D_1,\xi_1,\nu_1,\nu_2,\xi_2)\),  
\item \(G_0=\Mirror(d_1,d_2,D_1)\), and \item \(G=\Spread(G_0,\xi_2)\cap\mathbb{N}(\mu_3+1,\nu_2)\).
\end{itemize}     

\begin{lemma}
    \label{pp0eh2bd3}
   We have that  \(d_2-(\mu_3+1)<\xi_2\).
    \end{lemma}
    \begin{proof}        
    From the definition of a run, it follows that \(\xi_2\in\Period(\zz[d_1,d_2])\). Then from Lemma \ref{tudjkdi8545} we have that 
         \(\zz[d_1,d_2]=(p_1p_2)^{\alpha}p_1\), where \(p_1\in\Pal^*\), \(p_2\in\Pal^+\), \(\vert p_1p_2\vert=\xi_2\), and \(\alpha\in\mathbb{N}_0\). If \(d_2-(\mu_3+1)\geq \xi_2\) then \(\zz[d_1,d_2-\xi_2]=(p_1p_2)^{\alpha-1}p_1\in\Pal^+\). This contradicts to the definition of the minimal canonical palindrome \((d_1,d_2)\).
        The lemma follows. This ends the proof.
    \end{proof}

\begin{lemma}
    \label{duu8fkghg}
    We have that    
     \(\frac{d_1+d_2}{2}\geq \mu_3-\xi_1\) and \(\max(D_1)-d_1+1\leq \xi_2\). 
\end{lemma}
\begin{proof}
Suppose that \(\frac{d_1+d_2}{2}<\mu_3-\xi_1\). 
It follows then from Lemma \ref{uu78rbg68}  that \(\xi_1\in\Period(\zz[d_1,d_2])\), which is a contradiction to the definition of \(\mu_3\). Realize that \(\xi_1\not\in\Period(\zz[\min(D_1),\mu_3+1])\) and consequently \(\xi_1\not\in\Period(\zz[d_1,\mu_3+1])\), since \(\mu_3-d_1\geq c_1\xi_1\). We conclude that \begin{equation}
    \label{h7d6bvwgsd}
\frac{d_1+d_2}{2}\geq\mu_3-\xi_1\mbox{.}\end{equation}

Suppose that \(\xi_2<\max(D_1)-d_1+1\). It follows then from the definition of \(\NestPer_1(m)\) and (\ref{h7d6bvwgsd}) that \[\frac{d_1+d_2}{2}\geq \mu_3-\xi_1\geq \max(D_1)+c_1\xi_1-\xi_1\geq d_1+\xi_2-1+(c_1-1)\xi_1\mbox{.}\] This implies that \(\frac{1}{2}(d_2-d_1+1)\geq \xi_2-1+(c_1-1)\xi_1\). From (\ref{h7d6bvwgsd}) it follows that \(d_2-\mu_3+\xi_1\geq \frac{1}{2}(d_2-d_1+1)\) and consequently \(d_2-\mu_3+\xi_1\geq \xi_2-1+(c_1-1)\xi_1\). This is a contradiction to Lemma \ref{pp0eh2bd3}. We conclude that \(\xi_2\geq \max(D_1)-d_1+1\).
The lemma follows. This ends the proof.
\end{proof}

Let \(\widehat D_1=\PalExt(D_1,\xi_1,\nu_1,\nu_2,\xi_2)\).
We show that \(G\) contains all palindromic extension of \(D_1\) by the canonical palindromes of \((\nu_1,\nu_2,\xi_2)\).
\begin{proposition}
    \label{uu8ebn7g9} 
    We have that \(\sigma(\widehat D_1)\subseteq G\).    
\end{proposition}
\begin{proof}
Suppose \((n_1,n_2)\in\RunPal(\nu_1,\nu_2,\xi_2)\) such that \(n_1\in D_1\) and \(n_2>\mu_3\). 
From the definition of \(\minRunPal\) it is clear that \(n_1\in\mathbb{N}(d_1,d_2)\).
Let \(n_3=\Mirror(d_1,d_2,n_1)\).
From Lemma \ref{duu8fkghg} we have that \[n_1\leq \max(D_1)\leq \mu_3-c_1\xi_1<\mu_3-\xi_1\leq \frac{d_1+d_2}{2}\leq 
n_3\mbox{.}\] It follows from  Lemma \ref{hhyd89j33b} that \(n_2-n_3\equiv 0\pmod{\xi_2}\mbox{.}\) Because clearly \(n_3\in G_0\) we conclude that \(n_2\in G\). 
The proposition follows.
This ends the proof.
\end{proof}

The main theorem of this section shows that the palindromic extensions of \(D_1\in\varphi(\NestPer_1(m))\) by the canonical palindromes of the run \((\nu_1,\nu_2,\xi_2)\) are covered by an NPS cluster of degree \(m+1\).
\begin{theorem}
    \label{b88d93ji}
   We have that \(\omega(m+1,\widehat D_1)\leq 1\). 
\end{theorem}
\begin{proof}
Let \(\overline D_1=D_1\cap \mathbb{N}(d_1,\max(D))\).
 Realize  that \(G_0=\Mirror(d_1,d_2,D_1)=\Mirror(d_1,d_2,\overline D_1)\). 
 Property \ref{uj499b} implies that \((\overline D_1,\xi_1)\in\NestPer(m)\) and Property \ref{dy73bh2249} implies that \((G_0,\xi_1)\in \NestPer(m)\). 
 Obviously \(\diam(\overline D_1)=\diam(G_0)\). Hence Lemma \ref{duu8fkghg} implies that \(\diam(G_0)\leq \xi_2\). It follows from Lemma \ref{d8733j5r9} that \(\omega(m,G,\xi_2)\leq 2\leq \theta(m)\). 
In consequence we have that \((G,\xi_2)\in\NestPer(m+1)\mbox{.}\)
Proposition \ref{uu8ebn7g9} implies that \(\widehat D_1\subseteq G\). 
We conclude that \(\omega(m+1,\widehat D_1)\leq 1\).
This ends the proof. 
\end{proof}
\begin{remark}
\label{dhj8f7ekv2}
    Applying Property \ref{dy73bh2249} in the previous proof is the main reason for usage of padded words. Without padded words, we would need to derive an upper bound on \(\omega(m,\add(D,1))\), since concatenating of another palindrome to a given prefix of length \(n\) requires a shift to the position \(n+1\) and then looking for palindromic prefixes of \(\zz[n+1,\vert\zz\vert\).
\end{remark}

\subsection{Palindromic extensions of an NPS cluster}

The next lemma shows that every palindromic factor of \(\zz\) is a canonical palindrome of some run of \(\zz\). We omit the proof, as it is obvious.
\begin{lemma}
\label{dd793bfg}
    If \(n_1\leq n_2\in\mathbb{N}(1,\vert\zz\vert)\), \(\zz[n_1,n_2]\in\Pal^+\), and \[(\nu_1,\nu_2,\xi)=\pToRun(n_1,n_2)\]
    then \((n_1,n_2)\in\RunPal(\nu_1,\nu_2,\xi,n)\).
\end{lemma}

Given \(D\subseteq \mathbb{N}(1,\vert \zz\vert)\), \(n_1\leq n_2\in\mathbb{N}(1,\vert\zz\vert)\), and \(\zz[n_1,n_2]\in\Pal^+\), let \[\begin{split}
    \PalExt(D,n_1,n_2)=\{(n,p_1,p_2,\alpha)\in\PalExt(D)\mid n_1
    \leq n\leq n_2\mbox{ and }\\\Mirror(n_1,n_2,n)=\sigma(n,p_1,p_2,\alpha)\}\mbox{.}
\end{split}\]

Note that \(\sigma(\PalExt(D,n_1,n_2))=\Mirror(n_1,n_2,D)\). Then the following lemma is obvious. We omit the proof.
\begin{lemma}
\label{bbc7d933f}
If \(n\in\mathbb{N}(1,\vert\zz\vert)\), \(D\subseteq\mathbb{N}(1,\vert \zz\vert)\), and \((n_1,n_2)\in\CovPalAll(n)\) then there is 
     \((n_3,n_4)\in\CovPalEdge(n)\) such that \(\PalExt(D,n_1,n_2)\subseteq \PalExt(D,n_3,n_4)\).
\end{lemma}

We show that if \(m\in\Omega\) then the palindromic extensions of an NPS cluster of degree \(m\) can be covered by \((1+c_1+{4}\hh)\) NPS clusters of degree \(m+1\).
\begin{proposition}
\label{uibv78vg}
If \(m\in\Omega\), \((D,\xi)\in\NestPer(m)\), and \(\widehat D=\sigma(\PalExt(D))\) then 
\(\omega(m+1,\widehat D)\leq 1+c_1+{4}\hh\).
\end{proposition}
\begin{proof}
Let \(\mu_1=\min(D)\) and \(\mu_3=\perProlong(D,\xi)\).
Let \((D_1,D_2)=\separate(D,\xi)\), let \(\widehat D_1=\sigma(\PalExt(D_1))\), and let \(\widehat D_2=\sigma(\PalExt(D_2))\).
Let \(\widehat D_{1,1}=\widehat D_1\cap \mathbb{N}(\mu_1,\mu_3)\) and \(\widehat D_{1,2}=\widehat D_1\cap \mathbb{N}(\mu_3+1,\vert \zz\vert)\). 
Theorem \ref{f8u556ju} implies that \(\omega(m+1,\widehat D_{1,1})\leq 1\). Lemma \ref{hhbd628d} implies that \(\omega(m+1,\widehat D_{2})\leq c_1\).

Suppose that \(D_{1,2}\not=\emptyset\). Given \((n_1,n_2)\in\CovPalAll(\mu_3+1)\) and \((\nu_1,\nu_2,\xi_2)\in\Run\), 
let \[K(n_1,n_2)=\sigma(\PalExt(D_{1,2},n_1,n_2))\cap \mathbb{N}(\mu_3+1,\vert\zz\vert)\] and let 
\[K(\nu_1,\nu_2,\xi_2)=\sigma(\PalExt(D_{1,2},\xi_1,\nu_1,\nu_2,\xi_2)\mbox{.}\] 
Realize that \(K(\nu_1,\nu_2,\xi_2)\subseteq  \mathbb{N}(\mu_3+1,\vert\zz\vert)\).
From Lemma \ref{dd793bfg} and Lemma \ref{bbc7d933f}  it follows that 
    \[\begin{split}\widehat D_{1,2}= \bigcup_{(n_1,n_2)\in\CovPalAll(\mu_3+1)}K(n_1,n_2)= 
            \bigcup_{(n_1,n_2)\in\CovPalEdge(\mu_3+1)}K(n_1,n_2)=\\
        \bigcup_{\substack{(\nu_1,\nu_2,\xi_2)\in\pToRun(\CovPalEdge(\mu_3+1))}}K(\nu_1,\nu_2,\xi_2)
    \mbox{.} 
    \end{split}\]            
    From Corollary \ref{jud9ujorr} we have that 
    \(\vert \pToRun(\CovPalEdge(\mu_3+1))\vert\leq {4}\hh\mbox{.}\) 
    Theorem \ref{b88d93ji} implies that if \((\nu_1,\nu_2,\xi_2)\in \pToRun(\CovPalEdge(\mu_3+1))\) then \(\omega(m+1,K(\nu_1,\nu_2,\xi_2))\leq 1\). It follows that \(\omega(m+1,\widehat D_{1,2})\leq 4\hh\).
    Since obviously \(\widehat D=\widehat D_{1,2}\cup\widehat D_{1,2}\cup\widehat D_2\mbox{}\), the proposition follows.
    This ends the proof.
\end{proof}

We step to the main result of Section \ref{sec8903jhd}.
\begin{corollary}
\label{yysbn44df}
\(\Omega=\mathbb{N}_1\)
\end{corollary}
\begin{proof}
    The proof is by induction on \(m\in\mathbb{N}_1\). The base case for \(m=1\) follows from Lemma \ref{iodi9efjg}. Suppose that \(m-1\in\Omega\) and \(m\geq 2\). 
    Proposition \ref{uibv78vg} implies that \(m\in\Omega\), since \(c_3\hh\geq (1+c_1+{4})\hh\geq 1+c_1+{4}\hh\). The corollary follows. 
    This completes the proof.
\end{proof}

\section{Base positions}
\label{cbhd88e7jf}
Recall that \(\{\zz_1, \zz_2, \dots,\zz_{\hh}\}=\NPP(\zz)\), \(\zz=\zz_{\hh}\), and \(\vert \zz_1\vert<\vert \zz_2\vert<\dots<\vert \zz_{\hh}\vert\). Let \(\widehat \zz_1,\widehat \zz_2,\dots,\widehat \zz_{\hh-1}\in\Pal^+\) be such that \(\zz_{i+1}=\zz_i\widehat \zz_i\zz_i\), where \(i\in\{1,2,\dots,\hh-1\}\). 
\begin{remark}
Property \ref{uyue4d9bh} asserts that \(\widehat \zz_i\) exist and are uniquely determined.
\end{remark}

We define formally the base positions of \(\zz\) as mentioned in the introduction.
\begin{definition}
Let \(\BB_1=\{(1,1)\}\). Given \(j\in\mathbb{N}(2,\hh)\), let 
\[\begin{split}
    \widehat\BB_j&= \{(g,\overline e)\mid (g,e)\in\BB_{j-1}\mbox{ and }\overline e=e+\vert \zz_{j-1}\widehat \zz_{j-1}\vert\}\mbox{ and }\\\BB_j&=\widehat\BB_j\cup\BB_{j-1}\cup\{(j,1)\}\mbox{.}\end{split}\]
Let \(\BB=\BB_{\hh}\) and let \(\widetilde\BB=\{e\mid (g,e)\in\BB\}\). 

If \((g,e)\in \BB\) then we say that \(e\) is a \emph{base position} of \(\zz_g\).
\end{definition}
\begin{remark}
    Note that if \((g,e)\in\BB\) then \(\zz[e, e+\vert \zz_g\vert-1]=\zz_g\) and that \(\BB_j=\{(g,e)\in\BB\mid e+\vert \zz_g\vert-1\leq \vert \zz_j\vert\}\).
Also note that if \(e\in\widetilde\BB\) then \(\zz[e]=\bb\).
\end{remark}
\begin{example}
    Suppose \(a,c\in\Sigma_0\). Let \(\zz_1=\bb\), \(\zz_2=\bb a\bb\), and \(\zz_3=\bb a\bb a\bb c \bb a \bb a \bb\).
    Then \(\widehat\zz_1=a\), \(\widehat\zz_2=a\bb c\bb a\), \(\BB_1=\{(1,1)\}\), \(\widehat\BB_2=\{1,3\}\), \(\BB_2=\{(1,1), (1,3), (2,1)\}\), \(\vert \zz_2\widehat \zz_2\vert=8\), 
    \(\widehat \BB_3=\{(1,9), (1,11), (2,9)\}\), and \[\BB_3=\{(1,1), (1,3), (2,1), (1,9), (1,11), (2,9), (3,1)\}\mbox{.}\]
\end{example}
\begin{definition}
Suppose \(n_1<n_2\in\mathbb{N}(1,\vert \zz\vert)\). Let 
\begin{itemize}
\item \(\BB(n_1,n_2)=\{(g,e)\in\BB\mid n_1\leq e\leq e+\vert \zz_g\vert-1\leq n_2\}\mbox{,}\)
    \item If \(\BB(n_1,n_2)\not=\emptyset\) then let \(\height(n_1,n_2)=\max\{g\mid (g,e)\in\BB(n_1,n_2)\}\) and 
\(\width(n_1,n_2)=\max\{e_1+\vert \zz_{g_1}\vert-e_2\mid (g_1,e_1),(g_2,e_2) \in\BB(n_1,n_2)\}\mbox{.}\)
\item If \(\BB(n_1,n_2)=\emptyset\) then let \(\width(n_1,n_2)=\height(n_1,n_2)=0\). 
\end{itemize}
We call \(\height(n_1,n_2)\) and \(\width(n_1,n_2)\) the \emph{height} and the \emph{width} of \((n_1,n_2)\), respectively.
\end{definition}

We present four simple properties of base positions. We omit the proof. 
\begin{lemma}
\label{dy894309kff}
Let \(g\in\mathbb{N}(1,\hh-1)\), \(d\in\mathbb{N}_1\), \(e_1<e_2<\dots<e_d\in\mathbb{N}(1,\vert \zz\vert)\) be such that \((g,e_i)\in\BB\) for all \(i\in\mathbb{N}(1,d)\) and if \((g,e)\in\BB\) then \(e=e_i\) for some \(i\in\mathbb{N}(1,d)\). 
We have that: 
\begin{enumerate}[ref=Q\arabic*,label=Q\arabic*:]
\item \label{iieuf9ktr} If \(1\equiv i\pmod{2}\) then \((g+1,e_i)\in\BB\) and \(e_i+\vert \zz_{g+1}\vert -1=e_{i+1}+\vert \zz_{g}\vert-1\),
\item \label{dhu9878ejd}  if \(i<d\) and \(n_1\leq n_2\in\mathbb{N}(e_i+1, e_{i+1}-1)\) then \(\width(n_1,n_2)<\vert \zz_g\vert\),  
\item \label{hd887edj2s}  if \(i<d\) and \(n_1\leq n_2\in\mathbb{N}(e_i+\vert \zz_{g}\vert, e_{i+1}+\vert \zz_{g}\vert-2)\) then \(\width(n_1,n_2)<\vert \zz_g\vert\), and 
\item \label{i99kedvm} \(\bigcup_{i=1}^d\BB(e_i,e_i+\vert \zz_g\vert-1)=\{(j,e)\in\BB\mid j\leq g\}\).
\end{enumerate}
\end{lemma}

\begin{remark}
Less formally said, every second base position of \(\zz_{g}\) is also base position of \(\zz_{g+1}\) and if the height of an interval \((n_1,n_2)\) is strictly smaller than \(g\) and the interval \((n_1,n_2)\) overlaps with at most with one interval \((e_i,e_i+\vert\zz_g\vert-1)\) then the width of \((n_1,n_2)\) is strictly smaller than \(\vert\zz_g\vert\).
It is easy to see that the number of base positions of the palindrome \(\zz_g\) is equal to \(2^{\hh-g}\). These properties express the hierarchical structure of base positions mentioned in the introduction.
\end{remark}


We show that for any interval \((n_1,n_2)\) with the height \(\beta\), there are at most four intervals that contain the same base positions as \((n_1,n_2)\) and that have the width smaller or equal to \(\vert \zz_{\beta}\vert\). 
\begin{proposition}
    \label{id9e9ur6t}
    If \(n_1\leq n_2\in\mathbb{N}(1,\vert \zz\vert)\), 
\(\beta=\height(n_1,n_2)\), 
    then there is \({\YY}\subseteq \mathbb{N}^2(n_1,n_2)\), 
    such that \({\YY}\) has the following properties
\begin{enumerate}[ref=P\arabic*,label=P\arabic*:]
\item \label{yud83hjrf} \(\vert {\YY}\vert\leq {4}\),
\item \label{bv2ds4dfw} \(\width(n_3,n_4)\leq \vert \zz_{\beta}\vert\) for all \((n_3,n_4)\in{\YY}\), 
\item \label{udih7676rt} there is \((n_3,n_4)\in {\YY}\) such that \(\zz[n_3,n_4]=\zz_{\beta}\), and 
\item \label{dy9ej9rtik4} \(\BB(n_1,n_2)=\bigcup_{(n_3,n_4)\in {\YY}} \BB(n_3,n_4)\).
\end{enumerate}
\end{proposition}
\begin{proof}
Let \(H=\{(\beta,e)\in\BB(n_1,n_2)\}\). Clearly \({H}\not=\emptyset\), since \(\beta=\height(n_1,n_2)\). From Property \ref{iieuf9ktr} of Lemma \ref{dy894309kff} it follows that 
\(\vert H\vert \leq 2\). 

Let \({\YY}_1=\{(e,e+\vert\zz_{\beta}\vert-1)\mid (\beta,e)\in {H}\}\). Clearly if \((i_1,i_2)\in{\YY}_1\) then \(\width(i_1,i_2)=\vert \zz_{\beta}\vert\).  

Let \(n_3=\min\{e-1\mid (\beta,e)\in {H}\}\) and \(n_4=\max\{e+\vert \zz_{\beta}\vert\mid (\beta,e)\in{H}\}\). 
Let \({\YY}_2, \YY_3\subset \mathbb{N}^2(n_1,n_2)\) be such that 
\begin{itemize}
\item If \(n_1\leq n_3\) then \(\{(n_1,n_3)\}= {\YY}_2\) otherwise \(\YY_2=\emptyset\). 
\item If \(n_4\leq n_2\) then \(\{(n_4,n_2)\}= {\YY}_3\)  otherwise \(\YY_3=\emptyset\). 
\end{itemize}
Obviously if \((i_1,i_2)\in {\YY}_2\cup\YY_3\) then  \(\height(i_1,i_2)<\beta\).  
Property \ref{dhu9878ejd} implies that if \((i_1,i_2)\in {\YY}_2\) then \(\width(i_1,i_2)<\vert \zz_{\beta}\vert\) and Property \ref{hd887edj2s} implies that if \((i_1,i_2)\in {\YY}_3\) then \(\width(i_1,i_2)<\vert \zz_{\beta}\vert\).
Let \({\YY}={\YY}_1\cup{\YY}_2\cup\YY_3\). We have that \(\vert {\YY}_1\vert=\vert {H}\vert\leq 2\) and \(\vert {\YY}_2\vert,\vert {\YY}_3\vert\leq 1\).  It follows that \(\vert \YY\vert\leq 4\).

Property \ref{i99kedvm} implies that \(\BB(n_1,n_2)=\bigcup_{(n_3,n_4)\in {\YY}} \BB(n_3,n_4)\); realize that if \(\vert H\vert=2\), \((\beta,e_1),(\beta,e_2)\in{H}\), \(e_1<e_2\), then Property \ref{i99kedvm} asserts that \(\BB(e_1+\vert \zz_{\beta}\vert,e_2-1)=\emptyset\).
The lemma follows. This ends the proof.
\end{proof}

%
Let \[\begin{split}\UC=\{S\subseteq\mathbb{N}(1,\vert \zz\vert)\mid \mbox{ if } \mu_1\leq \mu_2\in\mathbb{N}(1,\vert \zz\vert)\mbox{ and }\xi\in\Period(\zz[\mu_1,\mu_2])\\\mbox{ then there is }\EE\subseteq\mathbb{N}(\mu_1,\mu_2)\mbox{ such that }\vert \EE\vert \leq c_2\\\mbox{ and }S\cap\bigcup_{\delta\in \EE}\mathbb{N}(\delta,\delta+\xi-1)=S\cap\mathbb{N}(\mu_1,\mu_2)\}\mbox{.}\end{split}\]

\begin{proposition}
\label{duid0jfgh}
We have that \(\widetilde\BB\in\UC\).
\end{proposition}
\begin{proof}
Suppose \(n_1\leq n_2\in\mathbb{N}(1,\vert \zz\vert)\) and \(\xi\in\Period(\zz[n_1,n_2])\).
Let \(\beta=\height(n_1,n_2)\) and let \(\YY\subseteq\mathbb{N}^2(n_1,n_2)\) be as in Proposition \ref{id9e9ur6t}.
Let \((n_3,n_4)\in \YY\) be as in Property \ref{udih7676rt}; it means that \(\zz[n_3,n_4]=\zz_{\beta}\). Since \(\order(\zz_{\beta})<2\) it follows that \(n_4-n_3+1<2\xi\).  Then Property \ref{bv2ds4dfw} implies that \(n_6-n_5+1<2\xi\) for every \((n_5,n_6)\in \YY\).
Let \(\EE=\{\delta_1,\delta_1+\xi\mid (\delta_1,\delta_2)\in \YY\}\). From Property \ref{dy9ej9rtik4} it is obvious that \(\widetilde\BB\cap\bigcup_{\delta\in \EE}\mathbb{N}(\delta,\delta+\xi-1)=\widetilde\BB\cap\mathbb{N}(n_1,n_2)\).
Property \ref{yud83hjrf} implies \(\vert \EE\vert \leq c_2\).  
We conclude that \(\BB\in\UC\). This completes the proof.
\end{proof}

\section{Covering base positions}
Given \(m\in\mathbb{N}_0\), 
let \(\psi(m)=\max\{\vert D\cap S\vert \mid D\in\varphi(\NestPer(m))\mbox{ and }S\in\UC\}\).
We present an upper bound on the function \(\psi\).
\begin{proposition}
\label{rr7tff5}
If \(m\in\mathbb{N}_0\) then 
\(\psi(m)\leq \lambda(\hh,m)\). 
\end{proposition}
\begin{proof}
It is clear that  that \(\psi(0)=1\). Hence 
the lemma holds for \(m=0\). Suppose \(m\geq 1\) and suppose that the lemma holds for \(m-1\).
Let \((D,\xi)\in\NestPer(m)\) and \(S\in \UC\). Let \(\mu_1=\min(D)\) and \(\mu_2=\max(D)\).
From the definition of \(\UC\) we have that there is \(\EE\in\mathbb{N}(\mu_1,\mu_2)\) such that \(\vert \EE\vert\leq c_2\) and \(S\cap\bigcup_{\delta\in \EE}\mathbb{N}(\delta,\delta+\xi-1)=S\cap\mathbb{N}(\mu_1,\mu_2)\).

Given \(\delta\in \EE\), let \(D_{\delta}= D\cap \mathbb{N}(\delta,\delta+\xi-1)\).
Property \ref{uj499b} and the definition of \(\NestPer(m)\) imply that \(\omega(m-1,D_{\delta})\leq \theta(m)\). 
Let \(D_{\delta,j}\in\varphi(\NestPer(m-1))\) be such that \(D_{\delta}=\bigcup_{j=1}^{\theta(m)}D_{\delta,j}\).

By induction \(\vert D_{\delta,j}\cap S\vert\leq \psi(m-1)\). It follows that  \(\vert D_{\delta}\cap S\vert\leq \theta(m)\psi(m-1)\). Then since \(\delta\in\mathbb{N}(1,c_2)\) we have that \(\psi(m)=\vert D\cap S\vert\leq c_2\theta(m)\psi(m-1)\).
Because \(\psi(0)=1\) it follows that  \[\begin{split}\psi(m)\leq c_2\theta(m)\psi(m-1)\leq c_2\theta(m)c_2\theta(m-1)\psi(m-2)\leq \dots\leq 
c_2^m\prod_{i=1}^{m}\theta(i)\leq \\
     c_2^m\prod_{i=1}^{m}(2c_1c_3\hh)^{i}\leq c_2^m(2c_1c_3\hh)^{\sum_{i=1}^{m} i}\leq c_2^m(2c_1c_3\hh)^{m^2}=\lambda(\hh,m)\mbox{.}
\end{split}\]

This completes the proof.
\end{proof}
We define that \(\Cover(0)=\{1\}\).
\begin{lemma}
\label{dhuf9ekkti}
If \(m\in\mathbb{N}_1\) then \(\Cover(m)\subseteq\sigma(\PalExt(\Cover(m-1)))\).
\end{lemma}
\begin{proof}
Let \(t\in\Sigma^*\) and \(p\in\Pal^+\) be such that \(t\bb p\bb\in\Prefix(\zz)\), \(\vert t\bb\vert\in\Cover(m-1)\), and \(\vert t\bb p\bb\vert\in\Cover(m)\). There is \((p_1,p_2)\in\PalCouple\) and \(\alpha\in\mathbb{N}_1\) such that \((p_1p_2)^{\alpha}p_1=\bb p\bb\). Let \(n=\vert t\bb\vert\). It follows that \((n,p_1,p_2,\alpha)\in\PalExt(n)\), \(n\in\Cover(m-1)\), and \(\sigma(n,p_1,p_2,\alpha)\in\Cover(m)\).
The lemma follows. This ends the proof.
\end{proof}
Let \(\widehat \Cover(m)=\omega(m,\Cover(m))\). We show that all the positions \(n\in\mathbb{N}(1,\vert \zz\vert)\) with the padded palindromic length of \(\zz[2,n-1]\) equal to \(m\) are covered by \((c_3\hh)^m\) NPS clusters of degree \(m\).
\begin{lemma}
\label{ry88dbh3v4}
    If \(m\in\mathbb{N}_0\) then \(\widehat \Cover(m)\leq (c_3\hh)^m\).
\end{lemma}
\begin{proof}
    We have that \(\widehat \Cover(0)= 1\), since \((\{1\},1)\in\NestPer(0)\). Thus the lemma holds for \(m=0\). By induction, 
    suppose \(m\geq 1\) and suppose that \begin{equation}        
    \label{tty3nbvk933}\widehat \Cover(m-1)\leq (c_3\hh)^{m-1}\mbox{.}\end{equation}
    Let \(M=\minNPSCov(m-1,\Cover(m-1)\)).
    From (\ref{tty3nbvk933}) we have that \(\vert M\vert\leq (c_3\hh)^{m-1}\). 
    From Lemma \ref{dhuf9ekkti}, it follows then that \[\Cover(m)\subseteq\sigma(\PalExt(\Cover(m-1)))\subseteq \bigcup_{(D,\xi)\in M}\sigma(\PalExt(D))\mbox{.}\]
    Corollary \ref{yysbn44df} implies that \(\omega(m,\sigma(\PalExt(D)))\leq c_3\hh\) for every \((D,\xi)\in M\). It follows that \(\omega(m,\Cover(m))\leq (c_3\hh)^m\). 
    The lemma follows. This completes the proof.
\end{proof}
\begin{corollary}
\label{djj887jejf}
Conjecture \ref{tue8eiru883} is true.
\end{corollary}
\begin{proof}
    From Proposition \ref{duid0jfgh} we have that \(\widetilde\BB\in\UC\). 
    From Lemma \ref{ry88dbh3v4} we have that there is \(M\subseteq \NestPer(\kk)\) such that \(\vert M\vert\leq \kk(c_3\hh)^{\kk}\) and \(\bigcup_{m=1}^{\kk}\Cover(m)\subseteq \widetilde\varphi(M)\).
    It follows then from Proposition \ref{rr7tff5} that  \[\vert\widetilde\BB\cap\bigcup_{m=1}^{\kk}\Cover(m)\vert\leq \vert\widetilde\BB\cap \widetilde\varphi(M)\vert\leq \kk(c_3\hh)^{\kk}\psi(\kk)\leq \kk(c_3\hh)^{\kk}\lambda(\hh,\kk)\mbox{.}\] Then (\ref{adhuf99j}) implies that \(\widetilde\BB\cap\bigcup_{m=1}^{\kk}\Cover(m)\subset \widetilde\BB\). 
    This is a contradiction to the assumption \(\maxPPL(\ww_0)\leq \kk\) and due to Lemma \ref{dd554hye88rc} it is a contradiction to \(\maxPL(\ww_0)\leq \kk\). The corollary follows. This completes the proof.
\end{proof}

\bibliographystyle{siam}
\IfFileExists{biblio.bib}{\bibliography{biblio}}{\bibliography{../!bibliography/biblio}}

\end{document}